%% file: main.tex
\documentclass{article}
\usepackage[english,russian]{babel}
\usepackage{arxiv}
\include{header}

\begin{document}

\title{Субградиентные методы для задач негладкой оптимизации с некоторой релаксацией условия острого минимума}
\author{
    Аблаев Сейдамет Серверович \\
    Московский физико-технический институт, \\
    Россия, 141701, Московская~обл., г. Долгопрудный, Институтский пер.,~9 \\
    Крымский федеральный университет им. В.\,И.\,Вернадского, \\
    Россия, 295007, Республика Крым, г.~Симферополь, проспект академика Вернадского,~4 \\
    \texttt{seydamet.ablaev@yandex.ru}
    \And
    Макаренко Дмитрий Владимирович \\
    Московский физико-технический институт, \\
    Россия, 141701, Московская~обл., г.~Долгопрудный, Институтский пер.,~9 \\
    \texttt{devjiu@gmail.com}
    \And
    Стонякин Федор Сергеевич \\
    Московский физико-технический институт, \\
    Россия, 141701, Московская~обл., г. Долгопрудный, Институтский пер.,~9 \\
    Крымский федеральный университет им. В.\,И.\,Вернадского, \\
    Россия, 295007, Республика Крым, г.~Симферополь, проспект академика Вернадского,~4 
    \\
    \texttt{fedyor@mail.ru}
    \And
    Алкуса С. Мохаммад \\
    Московский физико-технический институт, \\
    Россия, 141701, Московская~обл., г. Долгопрудный, Институтский пер.,~9 \\
    Национальный исследовательский университет «Высшая школа экономики», \\
    Россия, 101000, г. Москва, ул. Мясницкая, д.~20 \\
    \texttt{mohammad.alkousa@phystech.edu}
    \And
    Баран Инна Викторовна \\
    Крымский федеральный университет им. В.\,И.\,Вернадского, \\
    Россия, 295007, Республика Крым, г.~Симферополь, проспект академика Вернадского,~4
   \\
    \texttt{matemain@mail.ru}
}

\maketitle

\newpage

\begin{abstract}

Задачи негладкой оптимизации нередко возникают во многих приложениях. Вопросы разработки эффективных вычислительных процедур для негладких задач в пространствах больших размерностей весьма актуальны. В таких случаях хорошо применимы методы первого порядка (субградиентные методы), однако в достаточно общих ситуациях они приводят к невысоким скоростным гарантиям. Одним из подходов к этой проблеме может являться выделение подкласса негладких задач, допускающих относительно оптимистичные результаты о скорости сходимости в пространствах больших размерностей. К примеру, одним из вариантов дополнительных предположений может послужить условие острого минимума, предложенное в конце 1960-х годов Б.Т. Поляком. В случае доступности информации о минимальном значении функции для липшицевых задач с острым минимумом известен  субградиентный метод с шагом Б.Т. Поляка, который гарантирует линейную скорость сходимости по аргументу. Такой подход позволил покрыть ряд важных прикладных задач (например, задача проектирования на выпуклый компакт, или задача отыскания общей точки системы выпуклых множеств). Однако, как условие доступности минимального значения функции, так и само условие острого минимума выглядят довольно ограничительными. В этой связи в настоящей работе предлагается обобщённое условие острого минимума, аналогичное известному понятию неточного оракула. Предложенный подход позволяет расширить класс применимости субградиентных методов с шагом Б.Т. Поляка на ситуации неточной информации о значении минимума, а также неизвестной константы Липшица целевой функции. Более того, использование в теоретической оценке качества выдаваемого методом решения локальных аналогов глобальных характеристик целевой функции позволяет применять результаты такого типа и к более широким классам задач. Показана возможность применения предложенного подхода к сильно выпуклым негладким задачам и выполнено экспериментальное сравнение с известным оптимальным субградиентным методом на таком классе задач. Более того, получены результаты о применимости предложенной методики для некоторых типов задач с релаксациями выпуклости: недавно предложенное понятие слабой $\beta$-квазивыпуклости и обычной квазивыпуклости. Исследовано обобщение описанной методики на ситуацию с предположением о доступности на итерациях  $\delta$-субградиента целевой функции вместо обычного субградиента. Для одного из рассмотренных методов найдены условия, при которых на практике можно отказаться от проектирования итеративной последовательности на допустимое множество поставленной задачи. Исследования по обоснованию теорем 5, 6 и 7 выполнены за счет гранта Российского научного фонда и города
Москвы № 22-21-20065 (https://rscf.ru/project/22-21-20065/).
\end{abstract}

\newpage

\section{Введение}\label{intro}
Негладкие оптимизационные задачи часто возникают в самых разных приложениях. Однако известные методы субградиентного типа для таких задач приводят к пессимистичным теоретическим оценкам скорости сходимости в пространствах больших размерностей. Одним из подходов к этой проблеме может служить выделение специального класса задач с условием острого минимума \cite{6, 1}. Говорят, что $f$ удовлетворяет условию острого минимума, если
\begin{gather}\label{sm}
f(x) - f(x_*) \geq \alpha \min_{x_* \in X_*} \|x- x_*\|_2 \quad \forall x \in Q
\end{gather}
для некоторого фиксированного $\alpha >0$ и $f(x_*) = f^* = \min\limits_{x\in Q} f(x)$ для всякого $x_* \in X_*$, где $Q$ --- выпуклое и замкнутое подмножество $\mathbb{R}^n$, $X_*$ --- компакт и $\|\cdot\|_2$ --- евклидова норма.
При таком допущении удается предложить субградиентный метод с гарантией  линейной скорости сходимости в случае доступности информации о точном значении $f^*$ \cite{6} без использования с теоретических оценках скорости сходимости  параметра размерности пространства. Условие острого минимума верно, например, для задачи проектирования точки на выпуклый компакт. Однако требование доступности $f^*$ довольно ограничительно. В этой связи в настоящей работе рассмотрено некоторое обобщение условия острого минимума
\begin{gather}\label{eq_gen_sharp}
f(x) - \overline{f} \geq \alpha \min_{x_* \in X_*} \|x - x_* \|_2 - \Delta,
\end{gather}
где $\overline{f}$  --- это некоторое приближение минимального значения функции $f^*$, причём $\overline{f} \geq f^*$. Оказалось, что такое обобщение позволяет несколько расширить класс применимости субградиентных методов для задач с острым минимумом и шагом Б.Т. Поляка. Например, оно может покрыть постановку задачи с неточной информацией о $f^*$. При этом в работе выводятся оценки качества выдаваемого решения субградиентным методом с <<неточным>> аналогом шага Б.Т. Поляка на задачи с неизвестной константой Липшица целевой функции. Такой подход, связанный с использованием в теоретических результатах локальных аналогов глобальных характеристик целевой функции (в данном случае константа Липшица $f$) позволяет применять полученные результаты и к более широким классам задач с необязательно липшицевыми целевыми функциями.

В настоящей статье показана возможность применения предложенного подхода к сильно выпуклым негладким задачам (вообще говоря, не обладающим острым минимумом) и выполнено экспериментальное сравнение с известным оптимальным субградиентным методом на таком классе задач. Более того, получены результаты о применимости предложенной методики для некоторых типов задач с релаксациями выпуклости: недавно предложенное условие слабой $\beta$-квазивыпуклости \cite{7} или обычное условие квазивыпуклости (унимодальности). Такаже исследовано обобщение указанной методики на ситуацию с предположением о доступности на итерациях $\delta$-субградиента целевой функции вместо обычного субградиента. Для одного из рассмотренных методов найдены условия, при которых на практике можно отказаться от проектирования на допустимое множество поставленной задачи.

Итак, в работе рассмотрен класс задач
\begin{gather}\label{eq_1}
f(x)\rightarrow\min_{x\in Q},
\end{gather}
где $f$ --- непрерывная и выпуклая, или квазивыпуклая (унимодальная), или слабо $\beta$-квазивыпуклая при некотором $\beta\in (0;1]$ функция (понятие слабой $\beta$-квазивыпуклости введено недавно в \cite{7}). Для задач вида \eqref{eq_1} c предположением типа \eqref{eq_gen_sharp} в настоящей работе исследованы некоторые вариации субградиентного метода с шагом Б.Т. Поляка.

Данная статья состоит из введения, заключения и 6 основных разделов. В разделе 2 вводится аналог острого минимума (понятие $\Delta$-острого минимума) и рассматриваются примеры задач для которых это условие заведомо выполнено (примеры \ref{example1} --- \ref{example3}). При этом может не выполняться условие обычного острого минимума. В разделе 3 исследован адаптивный субградиентный метод на классе задач слабо $\beta$-квазивыпуклых функций, допускающих $\Delta$-острый минимум. В разделе 4 рассмотрен частично адаптивный субградиентный метод на классе квазивыпуклых функций в случае острого минимума, а также в случае $\Delta$-острого минимума. Получены оценки скорости сходимости рассматриваемых методов (теоремы \ref{theorem1} --- \ref{theorem4}), также исследовано поведение траектории субградиентного метода с целью получения возможности использовать метод без операции проектирования на допустимое множество. Получена оценка, описывающая удаление  траектории метода от начальной точки (теорема \ref{theorem5}). В разделе 5 исследуется аналог предложенного в разделе 4 метода на классе выпуклых липшицевых функций с использованием на итерациях $\delta$-субградиентов вместо обычных субградиентов. Использование $\delta$-субградиентов потенциально может позволить сэкономить вычислительные затраты метода за счёт отказа от требования доступности точного значения субградиента целевой функции в текущей точке. При этом доказано, что в оценке качества выдаваемого решения не накапливаются величины, соответствующие $\delta > 0$ (теорема \ref{theorem7}). В разделе 6 описаны вычислительные эксперименты по демонстрации эффективности и сравнению предложенной в разделе 3 методики для задач с $\Delta$-острым минимумом с работой известного субградиентного метода \cite{Bach_2012} для некоторых сильно выпуклых задач (выбор этого типа задач для сравнения обусловлен тем, что для них известны применимые на практике оценки качества решения по аргументу). Попутно также выведена адаптивная оценка качества выдаваемого субградиентным методом \cite{Bach_2012} решения. В разделе 7 на классе сильно выпуклых задач описана оценка качества решения субградиентного метода \cite{Bach_2012} с использованием адаптивно подбираемых параметров и для задачи о наименьшем покрывающем шаре описаны результаты экспериментов по сравнению эффективности предложеных в настоящей работе методов с шагом типа Б.Т. Поляка и субградиентного метода типа \cite{Bach_2012}.

\section{О некоторой релаксации понятия острого минимума}\label{sharp_fund}
Пусть $Q \subseteq \mathbb{R}^n$ --- это выпуклое замкнутое множество и $f$ --- вещественая выпуклая функция, определённая на $Q$. Говорят, что $x^*$ --- точка острого минимума $f(x)$ \cite{1,6}, если выполнено условие \eqref{sm}
для некоторого $\alpha > 0$ и $X_*$ --- это множество точных решений задачи минимизации $\min_{x\in Q} f(x)$, т.е. $X_* = \{x_* \in Q; f(x_*) = f^* \}$. Очевидно, что используя условие \eqref{sm} и зная оценку разности между $f(x)$ и $f^*$, можно оценить расстояние $\|x-x_*\|_2$. Также известно, что условие острого минимума \eqref{sm}  важно для получения линейной скорости сходимости субградиентного метода для негладких задач в пространствах больших размерностей \cite{6}.
В данной работе мы рассмотрим обобщенное условие острого минимума \eqref{eq_gen_sharp},
где параметры $\alpha>0$ и $\Delta > 0$ фиксированы и задано число $\overline{f}$ (в частности, его можно интерпретировать как приближение минимума $f^*$). Заметим, что условие \eqref{eq_gen_sharp} можно рассматривать и как обобщение острого минимума, в некотором смысле аналогичное неточному оракулу \cite{DevGleenerNesterov}. Приведём несколько примеров.

\begin{example}\label{example1}
Начнём с известной задачи о нахождении общей точки системы выпуклых компактов $\{\Omega_{i}\}_{i \in I = \overline{1, m}}$ в пространстве $\mathbb{R}^n$, которая может иметь смысл нахождения положения какой-нибудь изучаемой системы, соответсвующей ряду заявленных требований. Можно решать такую задачу методом последовательного проектирования, выбирая $x_0 \in \Omega_{1}$, а также при $i = \overline{0, m-1}$ в качестве $x_i \in \Omega_{i+1}$ точку, ближайшую к $x_i$. При этом каждая из $m$ подзадач будет сводиться к задаче минимизации функции вида $f_i(x)=\min\limits_{y \in \Omega_{i+1}} \|x - y\|_2$, которая удовлетворяет условию Липшица ($M = 1$) с острым минимумом ($\alpha = 1$) и начальной точкой $x = x_i$, $f_i^* = 0$ при условии, что пересечение множеств $\{\Omega_{i}\}_{i \in I}$ не пусто. В таком случае cубградиентный метод
$$
x_{k+1} = x_k - h_k \nabla f(x_k)
$$
с шагом Б.Т. Поляка \cite{6}
$$
h_k = \frac{f(x_k) - f^*}{\|\nabla f(x_k)\|_2^2}
$$
будет для каждой подзадачи выполнять проектирование точки $x_i$ на множество $\Omega_{i+1}$ за 1 итерацию.

Можно рассмотреть ситуацию, когда поставленная задача не разрешима и пересечение множеств $\{\Omega_{i}\}_{i \in I}$ пусто, но есть достаточно (удалённая не более, чем на $\Delta > 0$) близкая ко всем множествам (условно подходящая) точка. Тогда вполне логично рассматривать обобщение задачи об отыскании общей точки системы множеств с целью найти точку, достаточно близкую ко всем множествам. В этом случае можно также применять описанную выше схему. Но оптимальные значения $f_i^*$ уже в этом случае не известны. Тогда можно выбрать $\overline{f_i} \geqslant \Delta$ (на практике можно рассматривать и ситуацию $\overline{f_i} \leqslant \Delta$) и применять уже предлагаемые ниже в настоящей статье вариации субградиентных методов с заменой $f_i^*$ на $\overline{f_i}$ и полученные теоретические результаты c учётом условия $\Delta$-острого минимума. Отметим выполненные эксперименты с реализацией подхода указанного типа (см. замечание \ref{remark3} ниже).
\end{example}

Приведём ещё несколько примеров.

\begin{example}\label{ex_nonexact}
Пусть
$$
f(x) = \|x\|_2 + \gamma\|x - c\|_2^2
$$
при $x \in \mathbb{R}^n$ для некоторого фиксированного вектора $c \in \mathbb{R}^n$ малой нормы. Тогда $0 < f^* \leq f(0) = \gamma \|c\|_2^2$. Ясно, что при $\|x\|_2 > \gamma \|c\|_2^2$ вектор $x$ заведомо не совпадает с $x_*$, т.е. $x_*$ лежит в шаре с центром в $0$ радиуса  $\gamma \|c \|_2^2$.
В этом случае
$$
    f(x) \geq \|x - x_*\|_2 - \|x_*\|_2 + \gamma \|x - c\|_2^2
$$
и тогда
$$
    f(x) - f^* \geq \|x - x_* \|_2 - \|x_* \|_2 + \gamma \|x - c \|_2^2 - \gamma \|c \|_2^2,
$$
т.е.
$$
    f(x) - f^* \geq f(x) -  \gamma \|c \|_2^2 \geq \|x - x_*\|_2 - \|x_*\|_2 - \gamma \|c\|_2^2 \geq \|x - x_*\|_2 - 2\gamma \|c\|_2^2,
$$
и можно применять описанную далее методику для $\overline{f} = \gamma \|c \|_2^2$, $\alpha = 1$ и $\Delta = 2 \gamma \|c \|_2^2$.

С другой стороны, можно учитывать $2 \gamma$-сильную выпуклость и применять к задаче минимизации $f$, например, субградиентный метод \cite{Bach_2012}. Эксперименты для такого примера приведены далее в разделе 7.
\end{example}

\begin{example}\label{example3}
В качестве примера задачи, для которой верно обобщённое условие острого минимума, рассмотрим условие \textbf{слабого острого минимума вида}
\begin{gather}\label{weaksharp}
f(x) - f^{\ast} \geq \mu \|x-x_{\ast} \|_2^p, \quad \forall x \in Q
\end{gather}
при $p \in (1; \infty)$. При этом обычное условие острого минимума, вообще говоря, не верно.

В частности, при $p=2$ неравенство \eqref{weaksharp} верно для $2\mu$-сильно выпуклых функций.

Если $\|x - x_{\ast} \|_2 \geq \varepsilon$, то
$$
f(x) - f^{\ast} \geq \mu \varepsilon^{p - 1}\|x-x_* \|_2.
$$
В противном случае (т.е. $\|x-x_{\ast}\|_2 < \varepsilon$) имеем
$$
\mu \|x - x_{\ast} \|_2^p > \mu \|x - x_{\ast} \|_2 - \mu \varepsilon,
$$
откуда
$$
f(x) - f^{\ast} > \mu \varepsilon^{p-1} \|x-x_{\ast} \|_2 - \mu \varepsilon,
$$
т.е. выполнено условие обобщённого острого минимума \eqref{eq_gen_sharp} при
$$
\alpha = \mu \varepsilon^{p-1}, \;\,\, \Delta = \mu \varepsilon, \;\,\, \overline{f} = f^*.
$$

Более того, можно рассмотреть случай неточной информации о $f^*$ для функции $f$, удовлетворяющей \eqref{weaksharp}, т.е. когда известно лишь $\overline{f}$ такое, что $\overline{f} - f^* \leq \tilde{\Delta}.$ В этом случае условие обобщённого острого минимума \eqref{eq_gen_sharp} будет верно при
$$
\alpha = \mu \varepsilon^{p-1} \text{  и  } \Delta = \mu \varepsilon + \tilde{\Delta}.
$$
\end{example}

Ясно, что список примеров можно продолжить (например, отправляясь от примеров параграфа 3 статьи \cite{6}).

\section{Адаптивный метод для слабо $\beta$-квазивыпуклых задач с $\Delta$-острым минимумом}\label{adaptive_methods}

Пусть $Q \subseteq \mathbb{R}^n$ --- выпуклое замкнутое множество и $f: Q \rightarrow \mathbb{R}$ --- слабо $\beta$-квазивыпуклая функция для некоторого $\beta \in (0;1]$. Напомним \cite{7}, что $f$ называется {\it слабо $\beta$-квазивыпуклой} относительно точки минимума $x_{*}$ задачи \eqref{eq_1} на множестве $Q$, если для произвольного $x\in Q$ выполнено неравенство:
\begin{gather}\label{eqquasiconv}
f(x_{*})\geqslant f(x)+\frac{1}{\beta} \langle \nabla f(x), x_{*}-x \rangle,
\end{gather}
где $\nabla f(x)$~--- произвольный субградиент $f$ в точке $x$. Под субградиентом мы здесь и всюду далее понимаем элемент субдифференциала Кларка $f$ в точке $x$ и предполагаем его существование. Если $f$ дифференцируема в точке $x$, то под $\nabla f$ понимаем обычный градиент. Это вполне естественно для липшицевых функций (для существования субдифференциала Кларка в точке достаточно локальной липшицевости $f$ в окрестности этой точки). Если функция $f$ выпуклая, то субдифференциал Кларка совпадает с обычным субдифференциалом в смысле выпуклого анализа, и в таком случае условие слабой $\beta$-квазивыпуклости \eqref{eqquasiconv} верно при $\beta = 1$. Ясно, что если неравенство \eqref{eqquasiconv} верно для некоторого $\beta = \beta_0 \in (0; 1]$, то оно верно и при $\beta \in (0; \beta_0]$. Примеры функций, для которых возможно проверить свойство слабой $\beta$-квазивыпуклости и оценить параметр $\beta$, приведены в \cite{8}. В частности, это верно для невыпуклой функции $f(x)=|x|(1-e^{-|x|})$ при $\beta=1$ \cite{7}, которая имеет $\Delta$-острый минимум при достаточно малом $\Delta>0$. Отметим также, что субградиент $f$ может быть нулевым только в точке минимума: равенство $\nabla f(x)=0$ влечет $f(x)\leqslant f(x_*)$, т.е. $f(x)=f^{*}$.


Будем рассматривать задачи вида \eqref{eq_1},
где $Q$ --- выпуклое компактное подмножество $\mathbb{R}^{n}$, $f^* = f(x_*) = \min\limits_{x \in Q} f(x)$, а $f$ --- слабо $\beta$-квазивыпуклая функция при некотором $\beta\in (0;1]$.

Мы отправляемся от модификации субградиентного метода с шагом Б.Т. Поляка для задач минимизации слабо $\beta$-квазивыпуклой функции \cite{5}
\begin{gather}\label{1}
x_{k+1} = Pr_{Q} \{x_k - h_k \nabla f(x_k)\},
\end{gather}
где $Pr_{Q}$---~оператор проектирования на множество $Q$ и $\|\nabla f(x_k) \|_2 \neq 0 \; \text{при} \; k \geqslant 0$.

Предположим, что верно обобщенное условие острого минимума \eqref{eq_gen_sharp}, где параметры $\alpha>0$ и $\Delta > 0$ фиксированы и задано число $\overline{f} \geqslant 0$, $\overline{f} \geqslant f^*$ (его можно рассматривать как приближение минимума). Исследуем метод \eqref{1} c шагом
\begin{gather}\label{adaptive_step}
    h_k = \dfrac{\beta(f(x_k) - \overline{f})}{\| \nabla f(x_k) \|_2^2}.
\end{gather}

Cправедлива следующая
\begin{theorem}\label{theorem1}
Пусть $f$ --- слабо $\beta$-квазивыпуклая функция и для задачи \eqref{eq_1} с условием \eqref{eq_gen_sharp} используется метод \eqref{1} c шагом
$h_k = \dfrac{\beta(f(x_k) - \overline{f})}{\| \nabla f(x_k) \|_2^2}$. Пусть также $\forall i \geq 0$ верно $\alpha^2 \beta^2 \leq 2 \| \nabla f(x_i) \|_2^2$. Тогда верно неравенство:
\begin{gather}\label{adaptive_estimate}
    \begin{aligned}
    \min_{x_* \in X_*} \|x_{k+1} - x_* \|_2^2 \leq &  \prod_{i=0}^k \left ( 1 - \frac{\alpha^2\beta^2}{2 \| \nabla f(x_i) \|_2^2} \right ) \min_{x_* \in X_*} \|x_0 - x_* \|_2^2 + \\&
    \qquad \qquad \qquad \qquad + \sum_{i=0}^{k-1} \prod_{j=i+1}^k \left ( 1 - \frac{\alpha^2\beta^2}{2 \| \nabla f(x_j) \|_2^2} \right )\Delta_i + \Delta_k,
    \end{aligned}
\end{gather}
где $\Delta_k = \frac{\Delta^2}{2 \| \nabla f(x_k) \|_2^2}$ для всякого $k \geqslant 0$.
\end{theorem}
\begin{proof}
Имеем следующие неравенства (см. \cite{6}, соотношения (3) из доказательства теоремы 1):
$$
    2\beta h_k (f(x_k) - \overline{f}) \leq 2\beta h_k (f(x_k) - f^*) \leq h_k^2 \| \nabla f(x_k)\|_2^2 + \min_{x_* \in X_*} \|x_k - x_* \|_2^2 - \min_{x_* \in X_*} \|x_{k+1} - x_* \|_2^2.
$$
Это означает, что
$$
    \min_{x_* \in X_*} \|x_{k+1} - x_* \|_2^2 \leq \min_{x_* \in X_*} \|x_k - x_*\|_2^2 + h_k^2 \| \nabla f(x_k)\|_2^2 - 2\beta^2 h_k (f(x_k) - \overline{f}) =
$$
$$
    = \min_{x_* \in X_*} \|x_k - x_*\|_2^2 - \frac{\beta^2(f(x_k) - \overline{f})^2}{\| \nabla f(x_k) \|_2^2}.
$$
С учётом условия \eqref{eq_gen_sharp}  получим:
$$
    \min_{x_* \in X_*} \|x_{k+1} - x_* \|_2^2 \leq \left ( 1 - \frac{\alpha^2\beta^2}{\| \nabla f(x_k)\|_2^2} \right ) \min_{x_* \in X_*} \|x_k - x_* \|_2^2 + \frac{2\alpha \beta \Delta \|x_k - x_* \|}{\| \nabla f(x_k) \|_2^2} - \frac{\Delta^2}{\|\nabla f(x_k) \|_2^2} \leq
$$
$$
    \leq \min_{x_* \in X_*} \|x_k - x_* \|_2^2 \left ( 1 - \frac{\alpha^2\beta^2}{\|\nabla f(x_k) \|_2^2} + \frac{\alpha^2\beta^2}{2\|\nabla f(x_k) \|_2^2} \right) + \frac{2\Delta^2}{\|\nabla f(x_k) \|_2^2} - \frac{\Delta^2}{\|\nabla f(x_k) \|_2^2} \leq
$$
$$
    \leq \min_{x_* \in X_*} \|x_k - x_*\|_2^2 \left ( 1 - \frac{\alpha^2\beta^2}{2\| \nabla f(x_k) \|_2^2} \right ) +\frac{\Delta^2}{\| \nabla f(x_k) \|_2^2}.
$$
Пусть $\Delta_k = \frac{\Delta^2}{2 \| \nabla f(x_k) \|_2^2}$. Тогда
$$
    \min_{x_* \in X_*} \|x_{k+1} - x_* \|_2^2 \leq \min_{x_* \in X_*} \|x_k - x_*\|_2^2 \left ( 1 - \frac{\alpha^2\beta^2}{2 \| \nabla f(x_k) \|_2^2} \right ) + \Delta_k \leq
$$
$$
    \leq \left ( 1 - \frac{\alpha^2\beta^2}{2\| \nabla f(x_k) \|_2^2} \right ) \left (\min_{x_* \in X_*}\|x_{k-1}- x_* \|_2^2 \left ( 1 - \frac{\alpha^2\beta^2}{2 \| \nabla f(x_{k-1}) \|_2^2} \right ) + \Delta_{k-1} \right ) + \Delta_k \leq
$$
$$
    \leq \ldots \leq \prod_{i=0}^k \left ( 1 - \frac{\alpha^2\beta^2}{2 \| \nabla f(x_i) \|_2^2} \right ) \min_{x_* \in X_*} \|x_0 - x_* \|_2^2 + \sum_{i=0}^{k-1} \prod_{j=i+1}^k \left ( 1 - \frac{\alpha^2\beta^2}{2 \| \nabla f(x_j) \|_2^2} \right )\Delta_i + \Delta_k.
$$
\end{proof}

Однако, некоторая проблема предыдущего подхода заключается в потенциально возможных малых знаменателях, равных квадратам норм градиентов. Тогда накопление в оценке качества решения величин, соответствующих параметру $\Delta$, потенциально может оказаться существенным. Поэтому рассмотрим на классе $M$-липшицевых функций также метод \eqref{1} c шагом
\begin{gather}\label{nonadaptive_step}
h_k = \dfrac{\beta(f(x_k) - \overline{f})}{M^2}.
\end{gather}
Тогда верны неравенства
$$
    \min_{x_* \in X_*}\|x_{k+1} - x_* \|_2^2 \leq \min_{x_* \in X_*}\|x_k - x_*\|_2^2 \left ( 1 - \frac{\alpha^2\beta^2}{2 M^2} \right ) +\frac{\Delta^2}{M^2} \leq
$$
$$
    \leq \min_{x_* \in X_*}\|x_{k-1} - x_*\|_2^2 \left ( 1 - \frac{\alpha^2\beta^2}{2 M^2} \right )^2 + \frac{\Delta^2}{M^2} \left ( 1 + 1 - \frac{\alpha^2\beta^2}{2 M^2} \right) \leq
$$
$$
    \leq \min_{x_* \in X_*}\|x_{k-2} - x_*\|_2^2 \left ( 1 - \frac{\alpha^2\beta^2}{2 M^2} \right )^3 + \frac{\Delta^2}{M^2} \left ( 1 + 1 - \frac{\alpha^2\beta^2}{2 M^2} + \left ( 1 - \frac{\alpha^2}{2 M^2} \right )^2 \right) \leq \ldots \leq
$$
$$
    \leq \left ( 1 - \frac{\alpha^2\beta^2}{2 M^2} \right )^{k+1} \min_{x_* \in X_*}\|x_0 - x_*\|_2^2 + \frac{\Delta^2}{M^2} \sum_{i = 0}^k \left ( 1 - \frac{\alpha^2\beta^2}{2 M^2} \right )^i \leq
$$
$$
\leq \left ( 1 - \frac{\alpha^2\beta^2}{2 M^2} \right )^{k+1} \min_{x_* \in X_*}\|x_0 - x_*\|_2^2 + \frac{2\Delta^2}{\alpha^2\beta^2}.
$$
Таким образом, верна
\begin{theorem}\label{theorem2}
Пусть $f$ --- $M$-липшицева слабо $\beta$-квазивыпуклая функция и для задачи \eqref{eq_1} с условием \eqref{eq_gen_sharp} используется метод \eqref{1} c шагом $h_k = \dfrac{\beta(f(x_k) - \overline{f})}{M^2}$, причём  $\alpha^2\beta^2 \leq 2 M^2$. Тогда для всякого $k \geqslant 0$ верно неравенство
\begin{gather}\label{nonadaptive_estimate}
    \min_{x_* \in X_*}\|x_{k+1} - x_* \|_2^2 \leq \left ( 1 - \frac{\alpha^2\beta^2}{2 M^2} \right )^{k+1} \min_{x_* \in X_*}\|x_0 - x_*\|_2^2 + \frac{2\Delta^2}{\alpha^2\beta^2}.
\end{gather}
\end{theorem}
Отметим, то в задаче из примера \ref{example1} для каждой функции $\alpha=1$ и $M=1$. В этом случае использование адаптивного метода и соответствующие оценки теряют смысл и мы можем использовать теорему \ref{theorem2}, в которой чётко описан уровень влияния значение параметра $\Delta$ на качество выдаваемого методом решения.

\begin{corollary}
Если в условиях предыдущей теоремы положить $\Delta$ = 0 (обычный острый минимум), то метод \eqref{1} сходится с линейной скоростью:
\begin{gather}
\min_{x_* \in X_*}\|x_{k+1} - x_* \|_2^2 \leq \left (  1 - \frac{\alpha^2\beta^2}{M^2} \right)^{k+1} \min_{x_* \in X_*}\|x_0 - x_* \|_2^2.
\end{gather}
\end{corollary}
\begin{proof} Действительно,
$$
    \min_{x_* \in X_*} \|x_{k+1} - x_* \|_2^2 \leq \min_{x_* \in X_*} \|x_k - x_*\|_2^2 - \frac{\beta^2(f(x_k) - \overline{f})^2}{M^2}.
$$
С учётом условия \eqref{sm} предыдущее неравенство принимает вид:
$$
    \min_{x_* \in X_*} \|x_{k+1} - x_* \|_2^2 \leq \left ( 1 - \frac{\alpha^2\beta^2}{M^2} \right ) \min_{x_* \in X_*} \|x_k - x_* \|_2^2.
$$
По рекурсии получим:
$$
\min_{x_* \in X_*}\|x_{k+1} - x_* \|_2^2 \leq \left (  1 - \frac{\alpha^2\beta^2}{M^2} \right)^{k+1} \min_{x_* \in X_*}\|x_0 - x_* \|_2^2.
$$
\end{proof}

\section{Частично адаптивный субградиентный метод для квазивыпуклых задач: исследование поведения траектории и возможность отказа от операции проектирования}\label{partial_adaptive}

Здесь мы рассмотрим субградиентный метод для достаточно известного обобщения выпуклости. При этом указанная общность потребовала уже информацию о константе Липшица $M$ в случае, если целевая функция удовлетворяет условию Липшица. Рассмотрим задачу вида \eqref{1}, где $Q$~--- выпуклое замкнутое подмножество $\mathbb{R}^{n}$, $f^{*} = f(x_*) = \min\limits_{x \in Q} f(x)$, а функция $f$ квазивыпукла, т.е.
$$
    f(\lambda x+(1-\lambda)y)\leq\max\{f(x),f(y)\}\quad \forall\,x,y\in Q,\, \lambda\in[0;1].
$$
Предположим, что $f$ удовлетворяет условию Липшица с константой $M >0$:
\begin{gather}\label{eeq_2}
|f(x)-f(y)|\leq M \|x-y\|_{2}\quad\forall\,x,y\in Q.
\end{gather}

Введём (следуя ~\cite{3, 9}) вспомогательную величину
\begin{gather}\label{v}
\upsilon_{f}(x, x_{*}) :=
\begin{cases}
\left\langle\frac{\nabla f(x)}{\|\nabla f(x)\|_{2}},x-x_{*}\right\rangle, & \text{если $x\neq x_{*}$;} \\
0 & \text{при $x= x_{*}$.}
\end{cases}
\end{gather}
Если $\nabla f(x)=0$ при $x\neq x_{*}$, то вместо $\nabla f(x)$ можно использовать ненулевой вектор нормали $\mathcal{D}f(x)$ ко множеству уровня функции $f$ в точке $x$. Но для упрощения изложения далее сделаем допущение, что $\nabla f(x)\neq0$ при $x\neq x_{*}$.
Справедлива \cite{9} следующая
\begin{lemma}\label{lemma}
Если $f$ квазивыпуклая и $M$-липшицева функция, то для всякого $x \in Q$ верно неравенство
\begin{gather}\label{l1}
f(x) - f(x_*) \leq M \nu_f (x, x_*).
\end{gather}
\end{lemma}
Рассмотрим метод $(k=0,1,2,\ldots)$
\begin{gather}\label{eeq_5}
x_{k+1}=Pr_{Q}\{x_{k}-h_{k}\nabla f(x_{k})\},\quad \text{где}\quad h_{k}=\frac{f(x_{k})-f(x_{*})}{M\|\nabla f(x_{k})\|_{2}},
\end{gather}
где $M$ удовлетворяет \eqref{l1}.

\begin{theorem}\label{th5}
\cite{5} Пусть верно \eqref{l1} при некотором $0 < M < + \infty$ и $f$ имеет острый минимум. Тогда метод \eqref{eeq_5} сходится с линейной скоростью:
\begin{gather}\label{eeq_6}
\min_{x_* \in X_*}\|x_{k+1}-x_{*}\|^{2}_{2}\leq\left(1-\frac{\alpha^{2}}{M^{2}}\right)^{k+1} \min_{x_* \in X_*}\|x_{0}-x_{*}\|^{2}_{2}.
\end{gather}
\end{theorem}
Некоторым недостатком алгоритма \eqref{eeq_5} по сравнению с алгоритмом \eqref{1} может считаться требование знать $M$ из неравенства леммы \ref{lemma} при организации шагов. Однако эта же особенность позволяет предложить вариацию субградиентного метода для задач с некоторым $\Delta$-обобщением острого минимума вида \eqref{eq_gen_sharp}
при фиксированном значении $\Delta>0$ и заданном $\overline{f}\geq f^{*}$. В частности, данное условие логично использовать при отсутствии точной информации об $f^{\ast}$. Если выбрать в алгоритме \eqref{eeq_5} шаги вида (мы называем их <<частично>> адаптивными, т.к. используются как константа $M$, так и норма субградиента)
\begin{equation}\label{eq_partial_adaptive_step}
    h_{k}=\frac{f(x_{k})-\overline{f}}{M\|\nabla f(x_{k})\|_{2}},
\end{equation}
то получим соотношения (аналогично рассуждениям из доказательства теоремы \ref{th5})
$$
    \min_{x_{*}\in X_{*}}\|x_{k+1}-x_{*}\|_{2}^{2}\leq \min_{x_{*}\in X_{*}}\|x_{k}-x_{*}\|_{2}^{2}-\frac{1}{M^{2}}(f(x_{k})-\overline{f})^{2}\leq
$$
$$
    \leq \min_{x_{*}\in X_{*}}\|x_{k}-x_{*}\|_{2}^{2}\left(1-\frac{\alpha^{2}}{2M^{2}}\right)+\frac{\Delta^{2}}{M^{2}}\leq
$$
$$
    \leq \min_{x_{*}\in X_{*}}\|x_{0}-x_{*}\|_{2}^{2}\left(1-\frac{\alpha^{2}}{2M^{2}}\right)^{k+1}+\frac{\Delta^{2}}{M^{2}}\left(1+\left(1-\frac{\alpha^{2}}{2M^{2}}\right)+\ldots+\left(1-\frac{\alpha^{2}}{2M^{2}}\right)^{k}\right)\leq
$$
$$
    \leq \min_{x_{*}\in X_{*}}\|x_{0}-x_{*}\|_{2}^{2}\left(1-\frac{\alpha^{2}}{2M^{2}}\right)^{k+1}+\frac{2\Delta^{2}}{\alpha^{2}}.
$$

Некоторым недостатком алгоритма \eqref{eeq_5} по сравнению с алгоритмом \eqref{1} может считаться требование знать $M$ из неравенства леммы \ref{lemma} при выполнении алгоритма. Однако эта же особенность позволяет предложить вариацию субградиентного метода с $\Delta$-обобщением острого минимума вида \eqref{eq_gen_sharp}. Справедлива

\begin{theorem}\label{theorem4}
Пусть $M$-липшицева функция $f$ имеет $\Delta$-острый минимум при некотором $\Delta > 0$ и известно $\overline{f}\geq f^{*}$. Тогда для метода $x_{k+1}=Pr_{Q}\{x_{k}-h_{k}\nabla f(x_{k})\}$ с шагом вида \begin{equation}\label{step_teor4}
    h_{k}=\frac{f(x_{k})-\overline{f}}{M\|\nabla f(x_{k})\|_{2}}
\end{equation}
справедливо неравенство
\begin{equation}\label{teor4}
    \min_{x_{*}\in X_{*}}\|x_{k+1}-x_{*}\|_{2}^{2}\leq \min_{x_{*}\in X_{*}}\|x_{0}-x_{*}\|_{2}^{2}\left(1-\frac{\alpha^{2}}{2M^{2}}\right)^{k+1}+\frac{2\Delta^{2}}{\alpha^{2}}.
\end{equation}
\end{theorem}

Таким образом, наличие параметра $\Delta$ (связанного с рассматриваемым обобщением острого минимума) приводит к дополнительному слагаемому вида $O(\Delta)$ в оценке \eqref{teor4}. По-видимому, в случае малых значений $\|\nabla f(x_k)\|_2$ для алгоритма \eqref{1} такой вывод сделать уже нельзя и можно довольствоваться лишь оценкой \eqref{adaptive_estimate}.

Отметим, что на классе выпуклых функций неравенства  \eqref{nonadaptive_estimate} и \eqref{teor4} совпадают. Однако, функция $f$ в указанных неравенствах удовлетворяет разным предположеням об аналоге выпуклости. На практике сходимость метода с шагом \eqref{step_teor4} может оказаться лучше, чем у метода с шагом \eqref{nonadaptive_step}, поскольку возможна ситуация, когда $\| \nabla f(x_k) \|_2 < M$ при некоторых, что приводит к увеличению длины шага \eqref{step_teor4}.

\begin{remark}
Оказывается \cite{5}, что рассмотренный в этом разделе алгоритм \eqref{eeq_5} универсален в том смысле, что при некоторых условиях можно выписать теоретический результат о сходимости \eqref{eeq_5} со скоростью геометрической прогрессии для задач минимизации квазивыпуклых субдифференцируемых (т.е. локально липшицевых) гёльдеровых функций с острым минимумом и заранее известным точным значением $f^{*}$. Если $f$ удовлетворяет условию Гёльдера
\begin{gather}\label{gelder}
    | f(x) - f(y) | \leq M_{\nu} \|x - y \|_2^{\nu} \quad \forall x, y \in Q
\end{gather}
при некотором фиксированном $\nu \in [0; 1]$ и $0\leq M_{\nu} \leq + \infty$, то с учётом \eqref{sm} и \eqref{gelder} получаем, что для всякого $x\in Q$
$$
\alpha\min_{x_{*}\in X_{*}}\|x-x_{*}\|_{2}\leq f(x)-f^{*}\leq M_{\nu}\min_{x_{*}\in X_{*}}\|x-x_{*}\|_{2}^{\nu}.
$$
Поэтому
$$
\min_{x_{*}\in X_{*}}\|x-x_{*}\|_{2}^{1-\nu}\leq\frac{M_{\nu}}{\alpha},
$$
откуда при $\nu<1$ верно неравенство
$$
\min_{x_{*}\in X_{*}}\|x-x_{*}\|_{2}\leq \left(\frac{M_{\nu}}{\alpha}\right)^{\frac{1}{1-\nu}}.
$$
Поэтому при $0\leq\nu<1$ можно локализовать допустимую область $Q$, заменив её на пересечение исходного множества $Q$ с евклидовым шаром с центром в точке $x_{0}$ и радиусом $\displaystyle\left(\frac{M_{\nu}}{\alpha}\right)^{\frac{1}{1-\nu}}$. Если при этом $a > a_0 > 0$ не является достаточно малым, то для подходящей $M_{\nu}>0$ верно неравенство
$$
M_{\nu}a^{\nu}\leq\frac{M_{\nu}^{\frac{2}{1+\nu}}}{2}\frac{a^{2}}{\delta^{\frac{1-\nu}{1+\nu}}}+\frac{\delta}{2}\qquad\forall\,\delta>0.
$$
Такое неравенство позволит оценить
$$
M=\max\left\{M_{\nu},\frac{M_{\nu}^{\frac{2}{1+\nu}}}{2}\right\}
$$
при достаточно большом $\upsilon_{f}(x,x_{*})$ для некоторого $x_{*}\in X_{*}$ и применять метод \eqref{eeq_5}, если требования к искомой точности не очень высоки. В этом случае теорема \ref{th5} позволит гарантировать достижение сопоставимой с $a_0$ точности решения по аргументу за линейное время. Отметим, что описываемые рассуждения основаны на работе \cite{11}, где впервые было замечена возможность переноса оптимальной на классе выпуклых липшицевых задач оценки сложности субградиентного метода на классе задач минимизации квазивыпуклых субдифференцируемых функций в случае не очень высокой точности $\varepsilon \geq a_0 > 0$ приближённого решения. При этом для достаточно высокой точности $\varepsilon < a_0 $ такой вопрос тривиален вииду локальной липшицевости $f$.
\end{remark}

Рассмотрим ещё такую ситуацию. С точки зрения вычислительной сложности операция проектирования на допустимое множество может оказаться довольно затратной. Поэтому возникает вопрос об исследовании субградиентного метода без операции проектирования, а для этого необходимо описать удаление траекторий метода от начальной точки. Рассмотрим аналог метода \eqref{eeq_5} без операции проектирования $x_k$ на множество $Q$ и исследуем удаление траектории $\{x_k\}$ от начальной точки $x_0$
\begin{gather}\label{11}
x_{k+1} = x_k - h_k \nabla f(x_k), \; \text{   где  } \; h_k = \frac{f(x_k) - f^*}{M \| \nabla f(x_k) \|_2}.
\end{gather}

Это позволит, в частности, определить условия на $Q$, которые позволят применить метод \eqref{eeq_5} для задач с острым минимумом без дополнительной операции проектирования $x_k$ на $Q$. Если допустить, что $f$~--- $M$-липшицева при $M > 0$, то
\begin{gather*}
f(x_k) - f^* \leq  M \|x_k - x_*\|_2
\end{gather*}
для ближайшего к $x_k$ точного решения $x_*$: $f(x_*) = f^*$.

Оценим $\|x_N - x_0 \|_2$. Из теоремы \ref{th5} имеем
\begin{gather*}
\|x_k - x_* \|_2^2 \leq \left ( 1 - \frac{\alpha^2}{M^2} \right )^k \min_{x_* \in X_*} \|x_0 - x_*\|_2^2.
\end{gather*}
Поэтому
\begin{gather*}
\|x_{k+1} - x_k \|_2 = \frac{f(x_k) - f^*}{M} \leq \frac{1}{M} \left ( 1 - \frac{\alpha^2}{M^2} \right )^{\frac{k}{2}} \min_{x_* \in X_*} \|x_0 - x_*\|_2
\end{gather*}
\begin{gather*}
\begin{aligned}
\|x_N - x_0 \|_2 & \leq \sum_{k=0}^{N-1} \|x_{k+1} - x_k \|_2 \leq \frac{1}{M} \sum_{k=0}^{N-1} \left ( 1 - \frac{\alpha^2}{M^2} \right )^{\frac{k}{2}} \min_{x_* \in X_*} \|x_0 - x_*\|_2^2 \leq \\&
\leq \frac{1}{M \left ( 1- \sqrt{1 - \frac{\alpha^2}{M^2}} \right ) } \|x_0 -x_* \|_2 = \frac{M + \sqrt{M^2 - \alpha^2}}{\alpha^2} \|x_0 - x_* \|_2.
\end{aligned}
\end{gather*}

\begin{theorem}\label{theorem5}
Если $f$ удовлетворяет условию Липшица с постоянной $M > 0$ и условию острого минимума \eqref{sm} с постоянной $\alpha >0$, то для метода \eqref{11} справедливо неравенство
$$
    \|x_N - x_0 \|_2 \leq \frac{M + \sqrt{M^2 - \alpha^2}}{\alpha^2} \|x_0 - x_* \|_2.
$$
\end{theorem}

Таким образом, если допустимое множество $Q$ содержит шар с центром в точке $x_0$ радиуса $\frac{M + \sqrt{M^2 - \alpha^2}}{\alpha^2} \|x_0 - x_* \|_2$, то можно применять метод \eqref{11} без дополнительных операций проектирования $x_{k+1}$ ($k \geq 0$) на множество $Q$.

\section{Cубградиентный метод для задач с $\Delta$-острым минимумом с использованием на итерациях неточных $\delta$-субградиентов}\label{delta_subgrad}
Теперь допустим, что нам не доступна информация о точном значении субградиента $\nabla f$, используемая на итерациях алгоритма \eqref{1}. Пусть доступен лишь $\delta$-субградиент $f$ в текущей точке $x$, т.е.
\begin{gather}\label{ds}
f(y) - f(x) \geq \langle \nabla_{\delta} f(x), y-x \rangle - \delta \quad \forall y \in Q
\end{gather}
при некотором фиксированном $\delta > 0$ (если $\delta = 0$, то $\nabla_{\delta} f(x) = \nabla f(x)$) \footnote{Отметим, что это требование можно ослабить, предположив что при всяком $x \in Q$ неравенство $\langle \nabla_{\delta} f(x), y-x \rangle \geq 0$ влечёт $f(y) - f(x) \geq - \delta \quad \forall y \in Q$.}.

Например, для функции $f(x) = \max\limits_{z \in P} \varphi (x,z)$, $\varphi$ выпуклая по $x$ и непрерывная по $z \in P$ ($P$ --- некоторое компактное множество), $\nabla_{\delta} f(x)$ есть обычный субградиент функции $\varphi (x,z_{\delta})$:
$$\max\limits_{z \in P} \varphi (x,z) -\varphi (x,z_{\delta}) \leq \delta.$$

В качестве примера задачи, где можно использовать вместо обычного субградиента $\delta$-субградиент, рассмотрим задачу проектирования точки на выпуклый компакт $X_*$. Если точка $X$ далеко от множества $X_*$, то при нахождении субградиента функции расстояния от точки до множества можно пренебречь точностью направления единичного вектора субградиента и использовать $\delta$-субградиент вместо обычного субградиента. Отметим, что такой подход реализован далее экспериментально (см. замечание \ref{remark2}).

Перейдём к теоретическому обоснованию аналога метода с использованием $\delta$-субградиентнов. По аналогии с \eqref{v} введём вспомогательную величину

\begin{gather*}
\upsilon_{f}^{\delta} (x, x_{*}) := \left \langle\frac{\nabla_{\delta} f(x)}{\|\nabla_{\delta} f(x)\|_{2}},x-x_{*}\right\rangle
\end{gather*}
при $\nabla_{\delta} f(x) \neq 0$ и $\upsilon_{f}^{\delta} (x, x_{*}) = 0$ при $\nabla_{\delta} f(x) = 0$.

В предположении о том, что функция удовлетворяет условию Липшица с константой $M$, рассмотрим метод \eqref{1} с шагом
\begin{gather}\label{h}
h_k = \dfrac{f(x_k) - f^* - \delta}{M \|\nabla_{\delta} f(x_k)\|_2}.
\end{gather}
Тогда справедлив следующий результ об оценке скорости сходимости
\begin{theorem}
Пусть $f$ --- квазивыпуклая $M$-липшицева функция с острым минимумом \eqref{sm}, для которой в каждой точке $x \in Q$ можно вычислить $\delta$-субградиент \eqref{ds} и для задачи $f(x) \rightarrow \min\limits_{x \in Q}$ используется метод \eqref{1} с шагом \eqref{h}, а также $\alpha^2 \leq 2 M^2$. Тогда верно неравенство:
$$
    \min_{x_* \in X_*}\|x_{k+1}-x_{*}\|^{2}_{2} \leq \left(1-\frac{\alpha^{2}}{2M^{2}}\right)^{k+1} \min_{x_* \in X_*}\|x_{0}-x_{*}\|^{2}_{2} + \frac{2 \delta^2}{\alpha^2}.$$
\end{theorem}
\begin{proof} Справедливы следующие неравенства (см. доказательство теоремы \ref{theorem1}):
$$
    \min_{x_* \in X_*}\|x_{k+1}-x_{*}\|^{2}_{2}\leq\min_{x_* \in X_*}\|x_{k}-x_{*}\|^{2}_{2}+ h^{2}_{k}\|\nabla_{\delta} f(x_{k})\|^{2}_{2} - 2h_{k}\langle\nabla_{\delta} f(x_{k}),x_{k}-x_{*}\rangle=
$$
$$
    =\min_{x_* \in X_*}\|x_{k}-x_{*}\|^{2}_{2}+\frac{(f(x_{k})-f^* - \delta)^{2}}{M^{2}}- \frac{2(f(x_{k})-f^* - \delta)}{M}\cdot\left\langle\frac{\nabla_{\delta} f(x_{k})}{\|\nabla_{\delta} f(x_{k})\|_{2}},x_{k}-x_{*}\right\rangle=
$$
$$
    =\min_{x_* \in X_*}\|x_{k}-x_{*}\|^{2}_{2}+ \left(\frac{f(x_{k})-f^* - \delta)}{M}\right)^{2}-\frac{2(f(x_{k})-f^* - \delta)}{M} \cdot \upsilon_{f}^{\delta} (x_{k}, x_{*}) \leq
$$
\centerline{по лемме 1.1 из \cite{10}, имеем}
$$
    \leq \min_{x_* \in X_*}\|x_{k}-x_{*}\|^{2}_{2} - \left(\frac{f(x_{k})-f^* - \delta}{M}\right)^{2}.
$$
С учётом условия острого минимума \eqref{sm}:
$$
    \min_{x_* \in X_*}\|x_{k+1}-x_{*}\|^{2}_{2} \leq \min_{x_* \in X_*}\|x_{k}-x_{*}\|^{2}_{2} - \frac{\alpha^2}{M^2} \min_{x_* \in X_*}\|x_{k}-x_{*}\|_{2}^{2} - \frac{2 \alpha \delta}{M^2} \min_{x_* \in X_*}\|x_{k}-x_{*}\|_{2} - \frac{\delta^2}{M^2}.
$$
Заметим, что
$$
    \frac{2 \alpha \delta}{M^2} \min_{x_* \in X_*}\|x_{k}-x_{*}\|_{2} \leq \frac{\alpha^2}{2 M^2} \min_{x_* \in X_*}\|x_{k}-x_{*}\|_{2}^{2} + \frac{2 \delta^2}{M^2}.
$$
Тогда
$$
    \min_{x_* \in X_*}\|x_{k+1}-x_{*}\|^{2}_{2} \leq \left(1-\frac{\alpha^{2}}{2M^{2}}\right)\min_{x_* \in X_*}\|x_{k}-x_{*}\|^{2}_{2} + \frac{\delta^2}{M^2}.
$$
Далее, имеем
\begin{gather}\label{equation1}
    \begin{aligned}
    \min_{x_* \in X_*}\|x_{k+1}-x_{*}\|^{2}_{2} & \leq \left(1-\frac{\alpha^{2}}{2M^{2}}\right)^{k+1} \min_{x_* \in X_*} \|x_{0}-x_{*}\|^{2}_{2} + \left( \sum_{i=0}^k \left(1-\frac{\alpha^{2}}{2M^{2}}\right)^i \right) \frac{\delta^2}{M^2} = \\&
    = \left(1-\frac{\alpha^{2}}{2M^{2}}\right)^{k+1} \min_{x_* \in X_*}\|x_{0}-x_{*}\|^{2}_{2} + \frac{2 \delta^2}{\alpha^2}.
    \end{aligned}
\end{gather}
\qed
\end{proof}

Теперь рассмотрим метод \eqref{1} с шагом
\begin{gather}\label{hh}
h_k = \dfrac{f(x_k) - \overline{f} - \delta}{M \|\nabla_{\delta} f(x_k)\|_2}
\end{gather}
в предположении, что $f$ удовлетворяет условию $\Delta$-острого минимума \eqref{eq_gen_sharp}. Тогда
$$
    \min_{x_* \in X_*}\|x_{k+1}-x_{*}\|^{2}_{2} \leq \min_{x_* \in X_*}\|x_{k}-x_{*}\|^{2}_{2} - \left(\frac{f(x_{k})-f^* - \delta}{M}\right)^{2}
$$
и с учётом $\Delta$-острого минимума имеем
$$
    \min_{x_* \in X_*}\|x_{k+1}-x_{*}\|^{2}_{2}\leq\left(1-\frac{\alpha^{2}}{2M^{2}}\right) \min_{x_* \in X_*}\|x_{k}-x_{*}\|^{2}_{2} + \frac{(\Delta + \delta)^2}{M^2}.
$$
Поэтому по аналогии с \eqref{equation1}
$$
    \min_{x_* \in X_*}\|x_{k+1}-x_{*}\|^{2}_{2}\leq\left(1-\frac{\alpha^{2}}{2M^{2}}\right)^{k+1} \min_{x_* \in X_*}\|x_{0}-x_{*}\|^{2}_{2} + \frac{2(\Delta + \delta)^2}{\alpha^2}.
$$
Таким образом, верна
\begin{theorem}\label{theorem7}
Пусть $f$ --- выпуклая $M$-липшицева функция ($M > 0$), удовлетворяющая условию \eqref{eq_gen_sharp} такая, что в каждой точке доступен $\delta$-субградиент $f$ и для задачи минимизации $f$ на множестве $Q$ используется метод \eqref{1} с шагом \eqref{hh}, а также $\alpha^2 \leq 2 M^2$. Тогда верно неравенство:
$$
        \min_{x_* \in X_*}\|x_{k+1}-x_{*}\|^{2}_{2}\leq\left(1-\frac{\alpha^{2}}{2M^{2}}\right)^{k+1} \min_{x_* \in X_*}\|x_{0}-x_{*}\|^{2}_{2} + \frac{2(\Delta + \delta)^2}{\alpha^2}.
$$
\end{theorem}

Как видим, условие $\Delta$-острого минимума, а также использование на итерациях рассматриваемого метода $\delta$-субгриадиентов влечёт за собой дополнительное слагаемое $\frac{2(\Delta + \delta)^2}{\alpha^2}$ в оценке качества выдаваемого указанным методом решения.

\section{Адаптивная оценка скорости сходимости одного субградиентного метода для сильно выпуклых задач и вычислительные эсперименты для задачи о наименьшем покрывающем шаре}

Теперь перейдём к экспериментальному сравнению работы предложенных в разделе 3 подходов для задач с $\Delta$-острым минимумом с работой известного субградиентного метода \cite{Bach_2012} для некоторых сильно выпуклых задач. Выбор класса сильно выпуклых задач и метода \cite{Bach_2012} для сравнения обусловлен известными и  применимыми на практике теоретическими оценками качества приближённого решения по аргументу. Отметим адаптивный аналог \cite{Stonyakin_2021} теоретической оценки качества выдаваемого решения для субградиентного метода \cite{Bach_2012}. Напомним, что в работе рассматриваются задачи вида
\begin{gather}\label{min_q}
f(x)\rightarrow\min_{x\in Q},
\end{gather}
где $Q$ --- выпуклое замкнутое подмножество $\mathbb{R}^{n}$. Для субградиентного метода вида
\begin{gather}\label{orig}
x_{k+1} := Pr_{Q}\{x_k - h_k \nabla f(x_k) \}, \;\; \textit{где} \; h_k = \frac{2}{\mu (k+1)}
\end{gather}
известна следующая оценка скорости сходимости \cite{Bach_2012}:
\begin{equation}\label{orig_estimation_f}
f(\widehat{x}) - f(x_*) \leq \frac{2 M^2}{\mu (N+1)}  \; \text{  при   } \; \widehat{x} = \sum\limits_{k=1}^{N} \frac{2 k}{N (N+1)} x_k,
\end{equation}
где $M$ --- константа Липщица целевой функции $f$.

Оказывается, что данную оценку можно несколько улучшить на классе сильно выпуклых задач \cite{Stonyakin_2021}. Для доказательства понадобится следующая вспомогательная лемма.
\begin{lemma}\label{lemma_adapt}
Если $x_k$ и $x_{k+1}$ удовлетворяют \eqref{orig}, то для произвольного $x \in Q$ верно неравенство
$$
    h_k \langle \nabla f(x_k), x_k - x \rangle \leq \frac{h^2_k \|\nabla f(x_k)\|^2_2}{2} + \frac{1}{2}\| x - x_k\|^2_2 - \frac{1}{2}\| x -x_{k+1}\|^2_2.
$$
\end{lemma}
\begin{proof}
Непосредственно проверяются следующие неравенства:
\begin{gather*}
  \begin{aligned}
    h_k \langle \nabla f(x_k), x_k - x \rangle \leq h_k \langle \nabla f(x_k), x_k - x_{k+1} \rangle & + \frac{1}{2}\| x - x_k\|^2_2 -  \frac{1}{2}\| x - x_{k+1}\|^2_2 - \frac{1}{2}\| x_{k+1} - x_k\|^2_2 \leq \\
    \leq h_k \|\nabla f(x_k)\|_2 \| x_{k+1} - x_{k} \|_2 + \frac{1}{2}\| x - x_k\|^2_2 -  &\frac{1}{2}\| x -x_{k+1}\|^2_2 - \frac{1}{2}\| x_{k+1} -x_k\|^2_2 \leq \\ \leq \frac{h^2_k \|\nabla f(x_k)\|^2_2}{2} + \frac{1}{2}\| x - x_k\|^2_2& -  \frac{1}{2}\| x - x_{k+1}\|^2_2.
  \end{aligned}
\end{gather*}
\end{proof}

Согласно лемме \ref{lemma_adapt} получим, что для всяких $k \geq 0$ и $x \in Q$:
$$
    \langle \nabla f(x_k), x_k - x \rangle \leq \frac{h_k \|\nabla f(x_k)\|^2_2}{2} + \frac{1}{2 h_k}\| x - x_k\|^2_2 - \frac{1}{2 h_k}\| x -x_{k+1}\|^2_2.
$$
Далее, с учетом $\mu$-сильной выпуклости $f$:
$$
    f(x_k) - f(x)  + \frac{\mu}{2}\|x_k - x \|^2_2 \leq \langle \nabla f(x_k), x_k - x \rangle \;\; \forall x \in Q,
$$
откуда при всяком $k \geq 0$:
$$
    2k f(x_k) - 2k f(x) + k\mu\|x - x_k \|^2_2 \leq \frac{2k\|\nabla f(x_k)\|^2_2}{\mu(k+1)} + k(k+1)\frac{\mu}{2}\|x - x_k\|^2_2 - k(k+1)\frac{\mu}{2}\|x - x_{k+1}\|^2_2.
$$

Пусть алгоритм \eqref{orig} отработал $N$ шагов. Тогда можно просуммировать неравенства по $k$ от 1 до $N$ и учесть, что $\frac{k}{k+1} \leq 1$:
$$
    \sum_{k=1}^{N} 2k(f(x_k) - f(x) + \frac{\mu}{2} \| x_k - x \|^2_2) \leq \frac{2N\|\nabla f(x_k)\|^2_2}{\mu},
$$
откуда с учетом $2(1 + 2 +...+ N) = N(N + 1)$ получаем:
$$
    \sum_{k=1}^{N} \frac{2k}{N(N + 1)} (f(x_k) - f(x) + \frac{\mu}{2} \| x_k - x \|^2_2) \leq \frac{2\|\nabla f(x_k)\|^2_2}{\mu(N + 1)}.
$$

Ввиду выпуклости теперь $f$ верно неравенство
$$
     f(\widehat{x}) - f(x) \leq \frac{2}{\mu N (N + 1)} \sum_{k=1}^{N} \frac{k \| \nabla f(x_k)\|^2_2}{k + 1} \leq \varepsilon
$$
после $N = \mathcal{O}(\frac{M^2}{\mu\varepsilon})$ итераций алгоритма \eqref{orig}. Поэтому справедлива следующая
\begin{theorem}\label{ThmBachAdaptive}
Пусть $f$ --- $\mu$-сильно выпуклая функция. Тогда после $N$ итераций алгоритма:
$$
    x_{k+1} := Pr_{Q}\{x_k - h_k \nabla f(x_k) \}, \;\; \textit{где} \; h_k = \frac{2}{\mu (k+1)}
$$
будет верно неравенство:
\begin{equation}\label{adaptive_estimation_f}
    f(\widehat{x}) - f(x_*) \leq \frac{2}{\mu N (N+1)} \sum_{k=1}^{N} \frac{k \|\nabla f(x_k)\|_2^2}{k+1},
\end{equation}
где
$$
    \widehat{x} = \sum_{k=1}^{N} \frac{2 k}{N (N+1)} x_k.
$$
Если $f$ ещё и $M$-липшицева при $M >0$, то
$$
     f(\widehat{x}) - f(x) \leq \varepsilon
$$
после $N = \mathcal{O}(\frac{M^2}{\mu\varepsilon})$ итераций алгоритма \eqref{orig}.
\end{theorem}

Отметим, что если $x_*$ --- точное решение задачи минимизации $f$, то можно получить оценку скорости сходимости по аргументу вида
\begin{equation} \label{arg_est}
    \|\widehat{x} - x_*\|_2 \leq \frac{4}{\mu N (N+1)} \sum_{k=1}^{N} \frac{k \|\nabla f(x_k)\|_2^2}{k+1} \leq \frac{4M^2}{\mu(N+1)}.
\end{equation}

Полученный в теореме \ref{ThmBachAdaptive} результат применим и в случаях, когда константа Липщица ($M$) --- бесконечна или её значение сложно оценить. Более того, данный подход может быть распространён на важные прикладные задачи, среди которых задача бинарной классификации методом опорных векторов (SVM) \cite{Bach_2012}. По аналогии с работой \cite{Bach_2012} можно применять стохастический вариант зеркального спуска \eqref{orig}. Также отметим, что данные рассуждения и метод могут быть обобщены на класс вариационных неравенств, лагранжевых и седловых задач \cite{Stonyakin_2021}.

Для сравнения скорости сходимости метода \cite{Bach_2012} и полученной в теореме \ref{ThmBachAdaptive} оценки с предложенными в настоящей работе (раздел 3) вариациями субградиентных методов для задач с $\Delta$-острым минимумом проведены численные эксперименты для задачи о наименьшем покрытии точек шаром для $2$-сильно выпуклой функции
\begin{gather}\label{sphere_cover_strongly}
    f(x) := \max\left\{\|x - a_0\|_2^2, \|x - a_1\|_2^2, ..., \|x - a_m\|_2^2\right\},
\end{gather}
а также для не сильно выпуклой (но выпуклой) функции
\begin{gather}\label{sphere_cover}
    f(x) := \max\left\{\|x - a_0\|_2, \|x - a_1\|_2, ..., \|x - a_m\|_2\right\}.
\end{gather}

Начнём с иллюстрации преимуществ адаптивной оценки метода \cite{Bach_2012} из теоремы \ref{ThmBachAdaptive}. Будем рассматривать множество Q, которое равно евклидову шару с центром в 0. Начальная точка выбиралась случайно, но внутри Q. На рис. \ref{res_ex_strong_r5} и \ref{res_ex_strong_unlim} ниже показано поведение и характер убывания для оригинальной оценки (\ref{orig_estimation_f}) --- сплошная линия, адаптивной оценки (\ref{adaptive_estimation_f}) --- штрих-пунктирная линия и непосредственно невязки по функции и по аргументу соответственно --- штриховая линия. На рис. \ref{res_ex_strong_r5} показано поведение глобальной оценки, адаптивной и невязки по функции и аргументу в случае ограниченного $Q (R = 5)$. Случай, когда глобальной оценкой воспользоваться не получится, показан на рис. \ref{res_ex_strong_unlim}. В этом случае мы работаем на $Q = \mathbb{R}^n$, потому наблюдаем соотношение между адаптивной оценкой и непосредственно невязкой. Данные графики наглядно демонстрируют, насколько более точной может оказаться адаптивная оценка (\ref{adaptive_estimation_f}) для задачи \eqref{sphere_cover_strongly}.

\begin{figure}[H]
	\minipage{0.49\textwidth}
	\includegraphics[width=\linewidth]{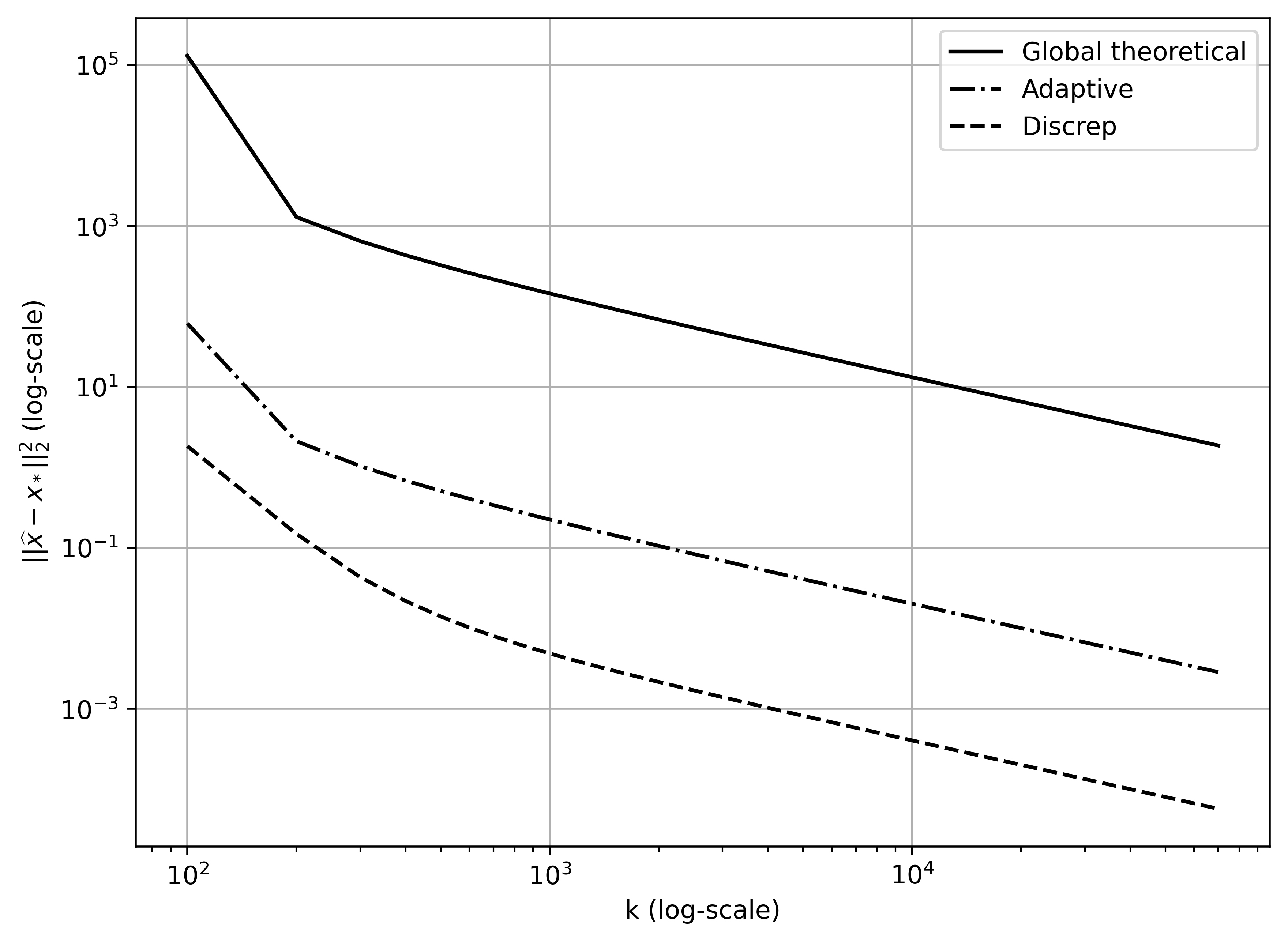}
	\endminipage\hfill
	\minipage{0.49\textwidth}
	\includegraphics[width=\linewidth]{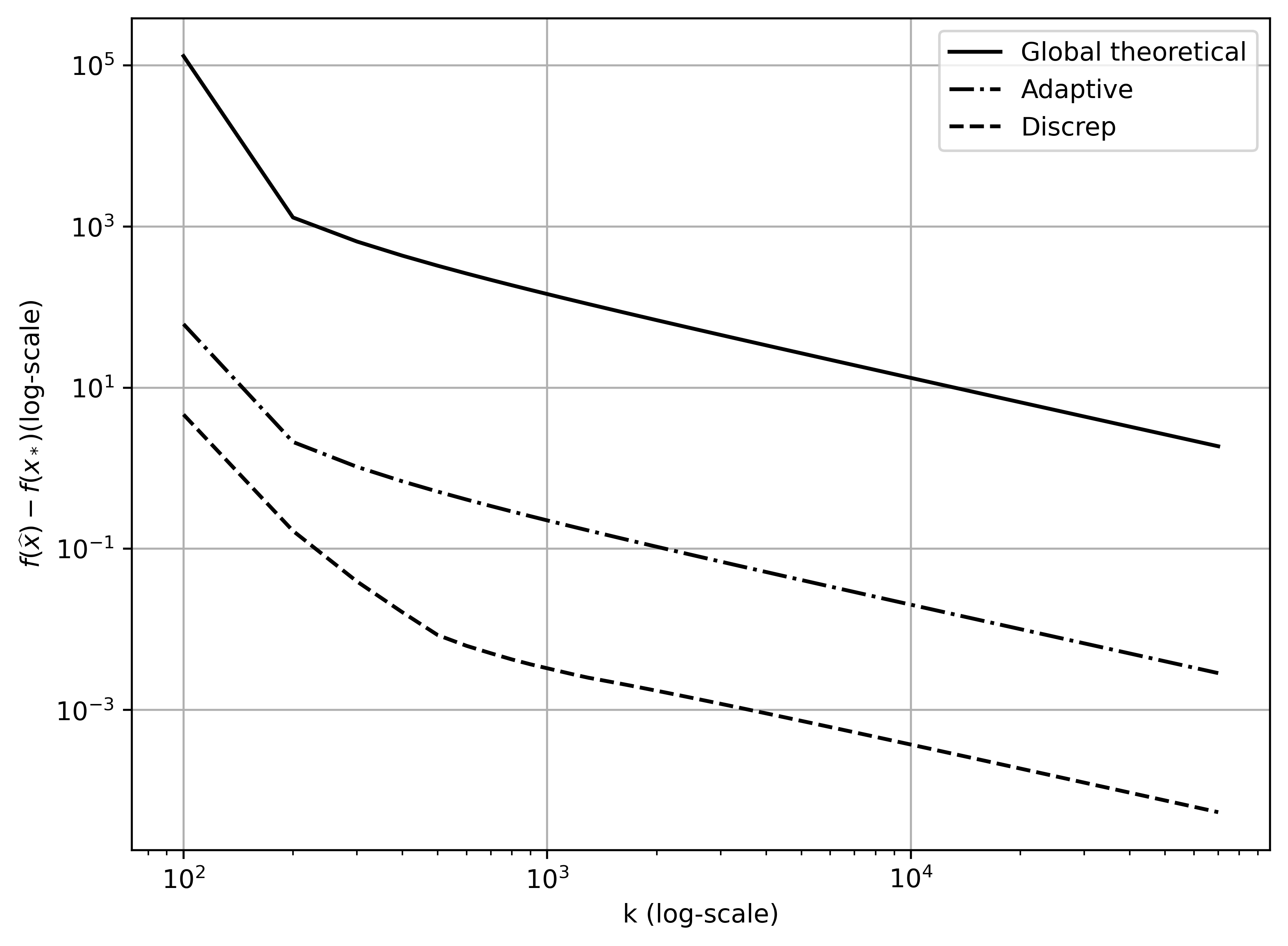}
	\endminipage\hfill
	\caption{Результаты решения задачи минимизации \eqref{sphere_cover_strongly}, где  $n= 1\,000, r = 5$ и  шар $Q$ радиуса 4.}
	\label{res_ex_strong_r5}
\end{figure}

\begin{figure}[H]
	\minipage{0.49\textwidth}
	\includegraphics[width=\linewidth]{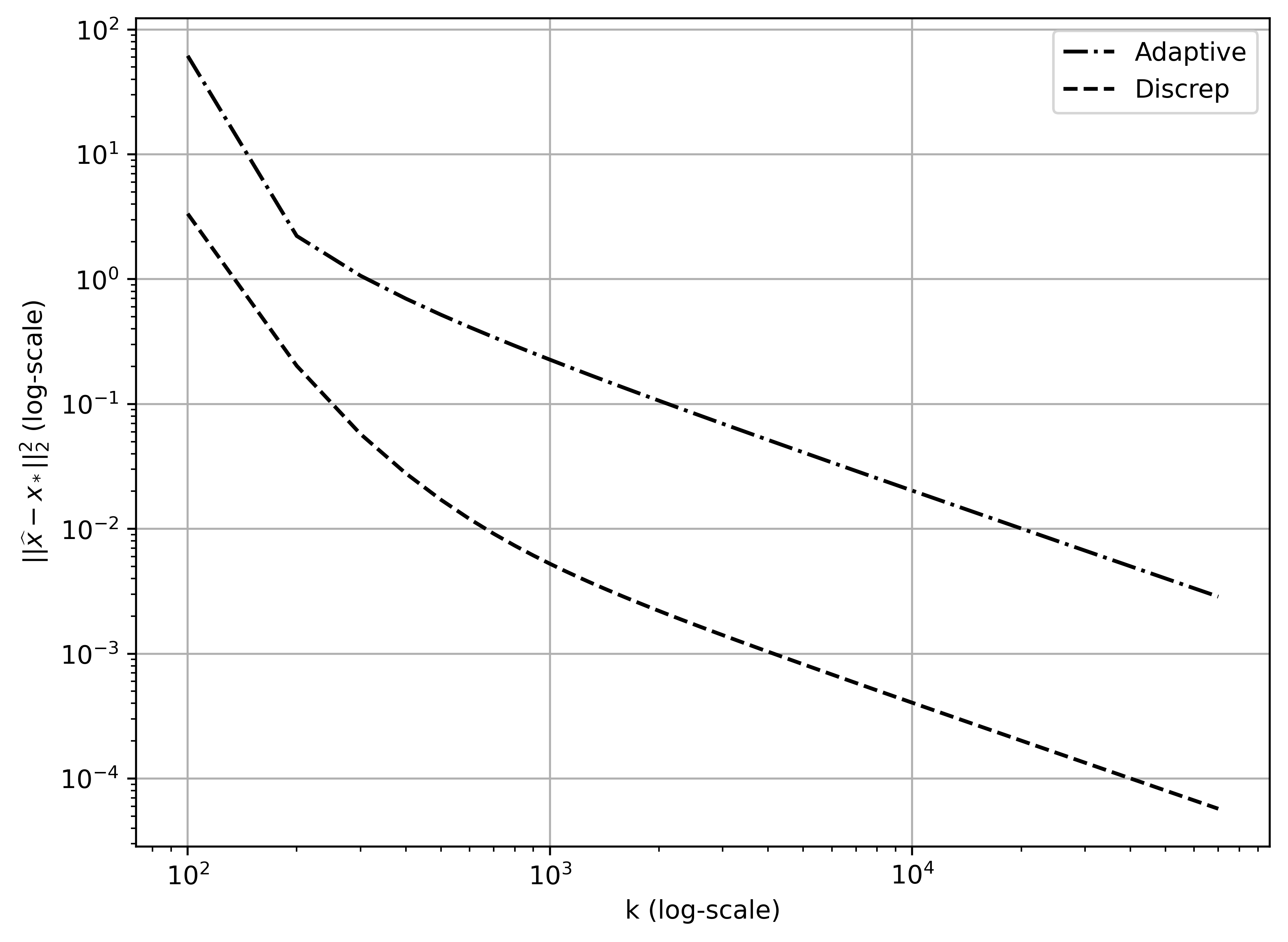}
	\endminipage\hfill
	\minipage{0.49\textwidth}
	\includegraphics[width=\linewidth]{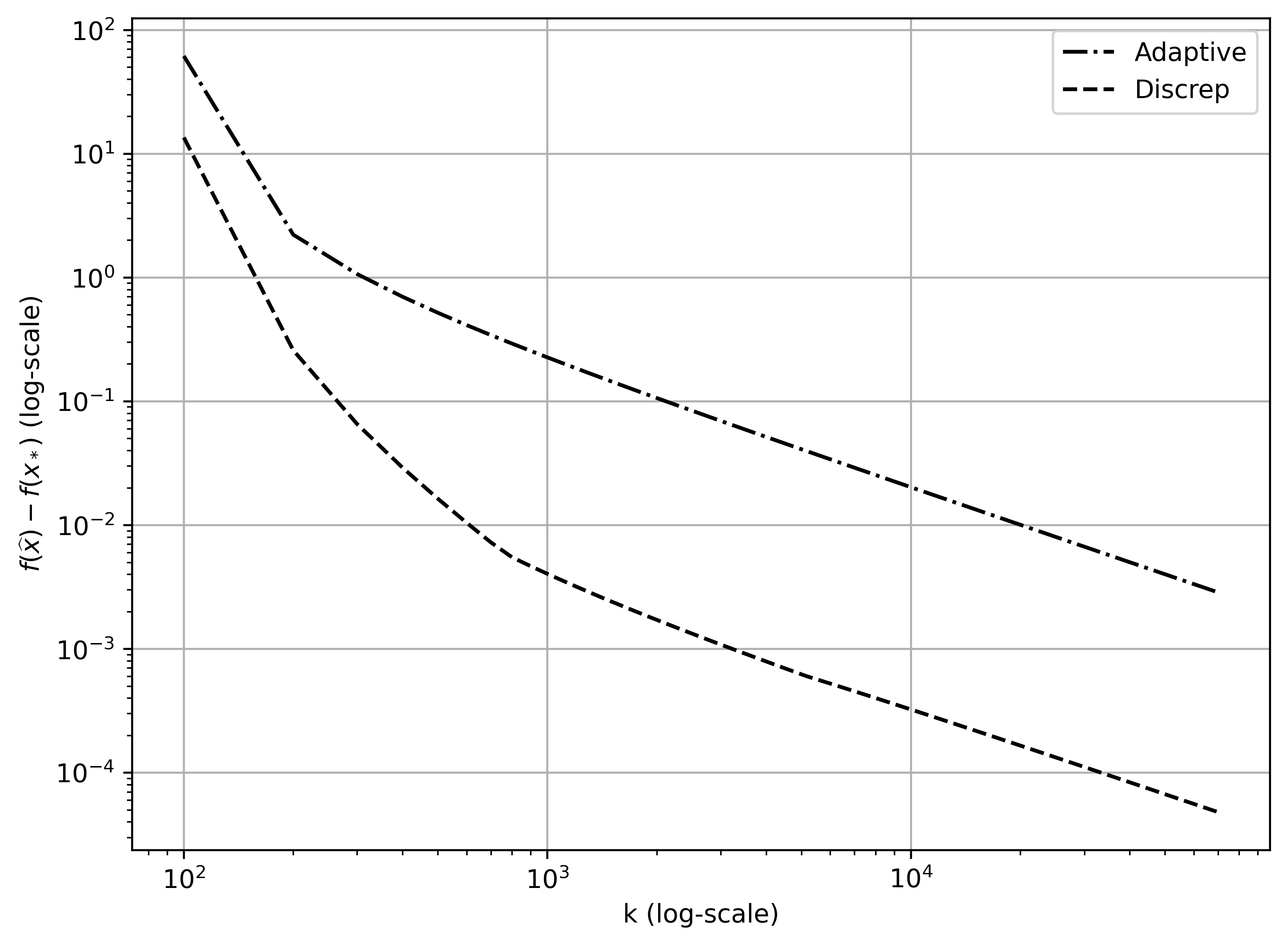}
	\endminipage\hfill
	\caption{ Результаты решения задачи минимизации \ref{sphere_cover_strongly}, где  $n= 1\,000, r = 5$ и  $Q = \mathbb{R}^n$.}
	\label{res_ex_strong_unlim}
\end{figure}

Теперь перейдём к выпуклой постановке \eqref{sphere_cover} с целью исследования эффективности предложенных в разделе 3 субградиентных методов с $\Delta$-острым мнимимумом. К существующему набору точек, представленных для покрытия, с известным значением центра добавим дополнительную точку, которая находится вне исходного шара достаточно близко к границе (удалена не более, чем на $\Delta > 0$). Данный подход позволяет оценить <<приближённое>> значение минимума $\overline{f}$, что позволит применить разработанные выше варианты субградиентного метода с $\Delta$-острым минимумом. При этом новое значение минимума останется внутри исходной сферы. Поскольку оптимальное значение функции --- это радиус искомого шара, покрывающего все точки, а $x_*$ всегда будет расположена внутри него, то для всякого $x$ верно неравенство $ f(x) \geq \| x - x_*\|_2$. Рассмотрим целевую функцию вида
\begin{gather}\label{allpha_sphere_cover}
    f(x) := \alpha \max \{\|x - a_0\|_2, \|x - a_1\|_2, ..., \|x - a_m\|_2\}.
\end{gather}
Тогда значение $\Delta$ можно оценить  из (\ref{eq_gen_sharp}):
    $f(x) - \overline{f} \geq \alpha\|x- x_*\|_2 - \Delta, \quad \Delta \geq \overline{f}$.

Отметим, что данная постановка значительно влияет на величину теоретической оценки качества решения (\ref{adaptive_estimate}) для метода \eqref{1}.
Наиболее значимый вклад в оценку (\ref{adaptive_estimate}) дает последнее слагаемое $\frac{\Delta^2}{2\|\nabla f(x_k)\|^2_2}$, причём
$     \Delta \sim \overline{f} \sim \alpha \|\overline{x}-a\|_2 $ и
$     \|\nabla f(x_k)\|_2 = \alpha $. Поэтому последнее слагаемое пропорционально радиусу шара, соответсвующему <<приближённому>> решению. Это и подтверждается экспериментально. Для сравнения, ниже на рис. \ref{res_sharp_convex} и \ref{res_strong_convex} приведены результаты работы для того же набора входных точек, которые необходимо покрыть в обоих постановках --- (\ref{allpha_sphere_cover}) и (\ref{sphere_cover_strongly}). Начальная точка также одна и та же. Сравниваются методы \eqref{1} и \eqref{orig}. Первый из этих методов обеспечивает сходимость буквально за 10 итераций к <<приближённому>> решению с заданной точностью и даже позволяет эту точность повысить. Второй же метод достигает схожих (с геометрической точки зрения) результатов за значительно большее количество итераций, однако он позволяет повышать точность приближённого решения на дальнейших итерациях.

Подтверждение данного теоретического наблюдения хорошо иллюстрируется на рис. \ref{res_sharp_convex} и \ref{res_strong_convex}. На рис. \ref{res_sharp_convex} показано поведение субградиентного спуска, использующего $\Delta$-острый минимум (теорема \ref{theorem4}), а именно --- быстрая сходимость к <<приближенному>> решению. Штрих-пунктирная линия соответствует оценке \eqref{eq_gen_sharp}, а штриховая --- невязке по функции и аргументу. На рис. \ref{res_strong_convex} показано поведение метода для той же задачи, но с использованием сильно выпуклого целевого функционала (теорема \ref{ThmBachAdaptive}). Скорость убывания уже не столь высокая, но точность получаемого решения в итоге выше. Сплошная линия --- это глобальная оценка \eqref{orig_estimation_f}, штрих-пунктирная --- адаптивная \eqref{adaptive_estimation_f}, а штриховая --- невязка по функции и аргументу.

\begin{figure}[H]
	\minipage{0.49\textwidth}
	\includegraphics[width=\linewidth]{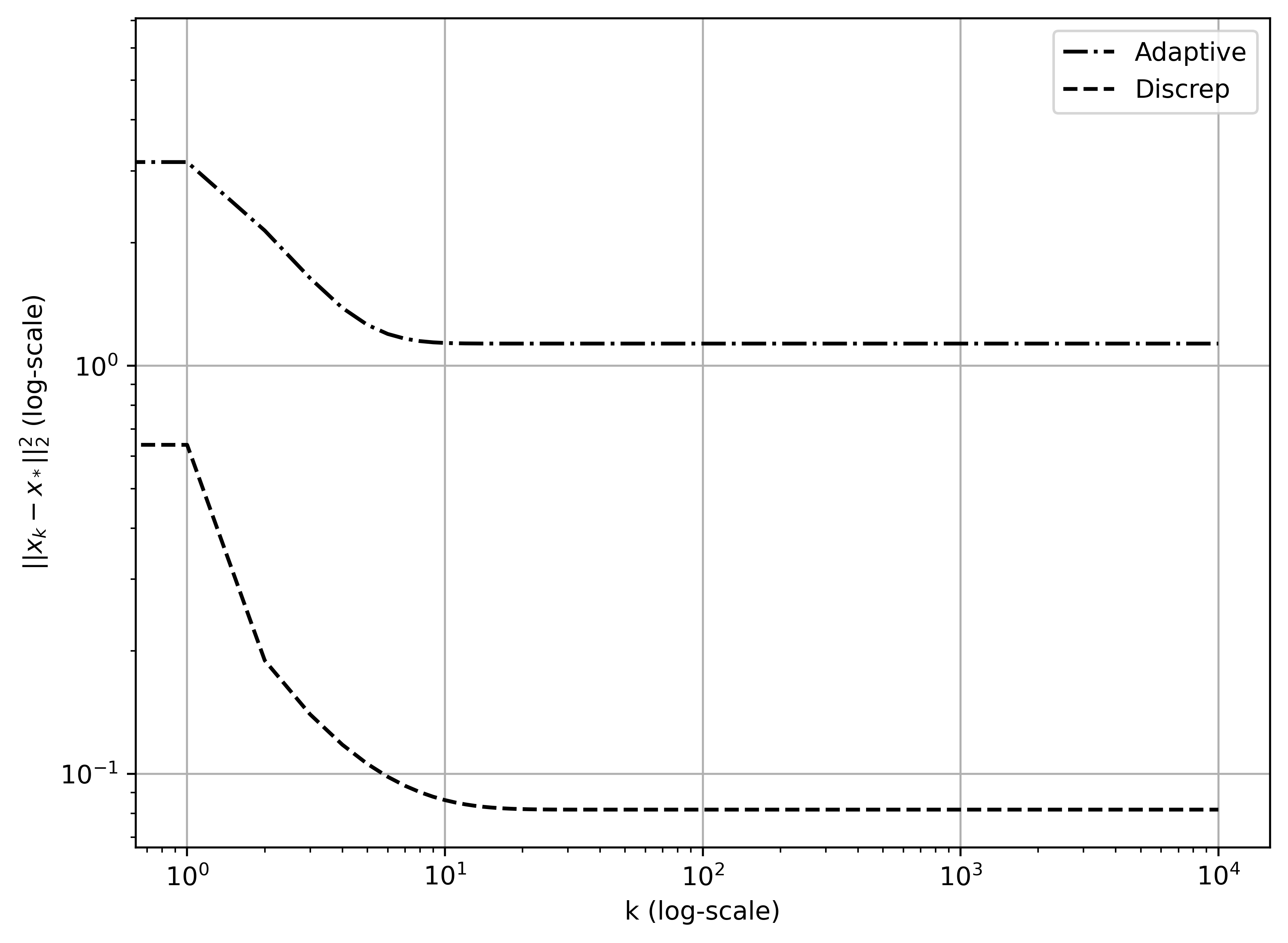}
	\endminipage\hfill
	\minipage{0.49\textwidth}
	\includegraphics[width=\linewidth]{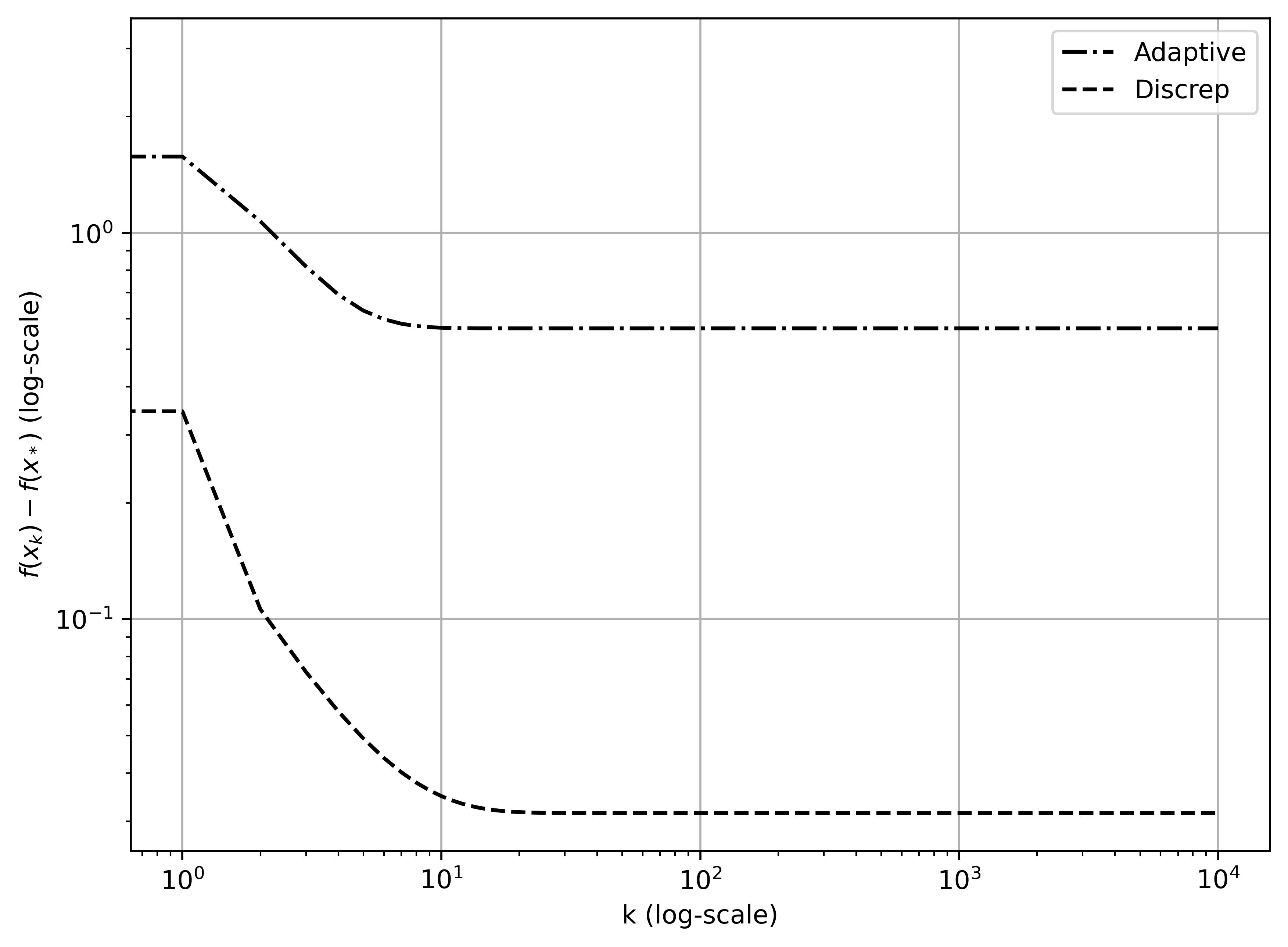}
	\endminipage\hfill
	\caption{ Результаты решения задачи минимизации (\ref{allpha_sphere_cover}), где  $n= 1\,000, r = 0.7525, \alpha = 0.6$.}
	\label{res_sharp_convex}
\end{figure}

\begin{figure}[H]
	\minipage{0.49\textwidth}
	\includegraphics[width=\linewidth]{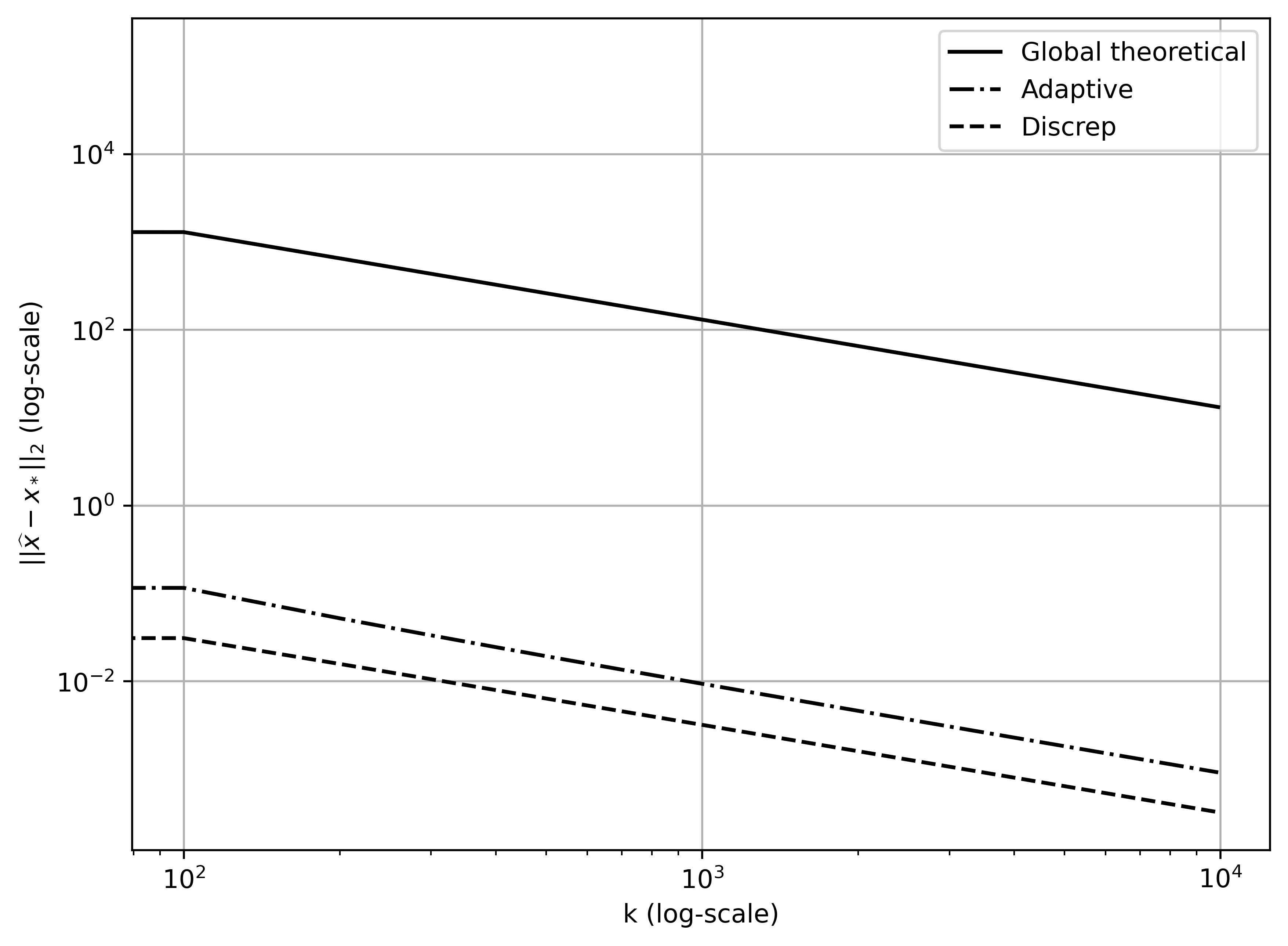}
	\endminipage\hfill
	\minipage{0.49\textwidth}
	\includegraphics[width=\linewidth]{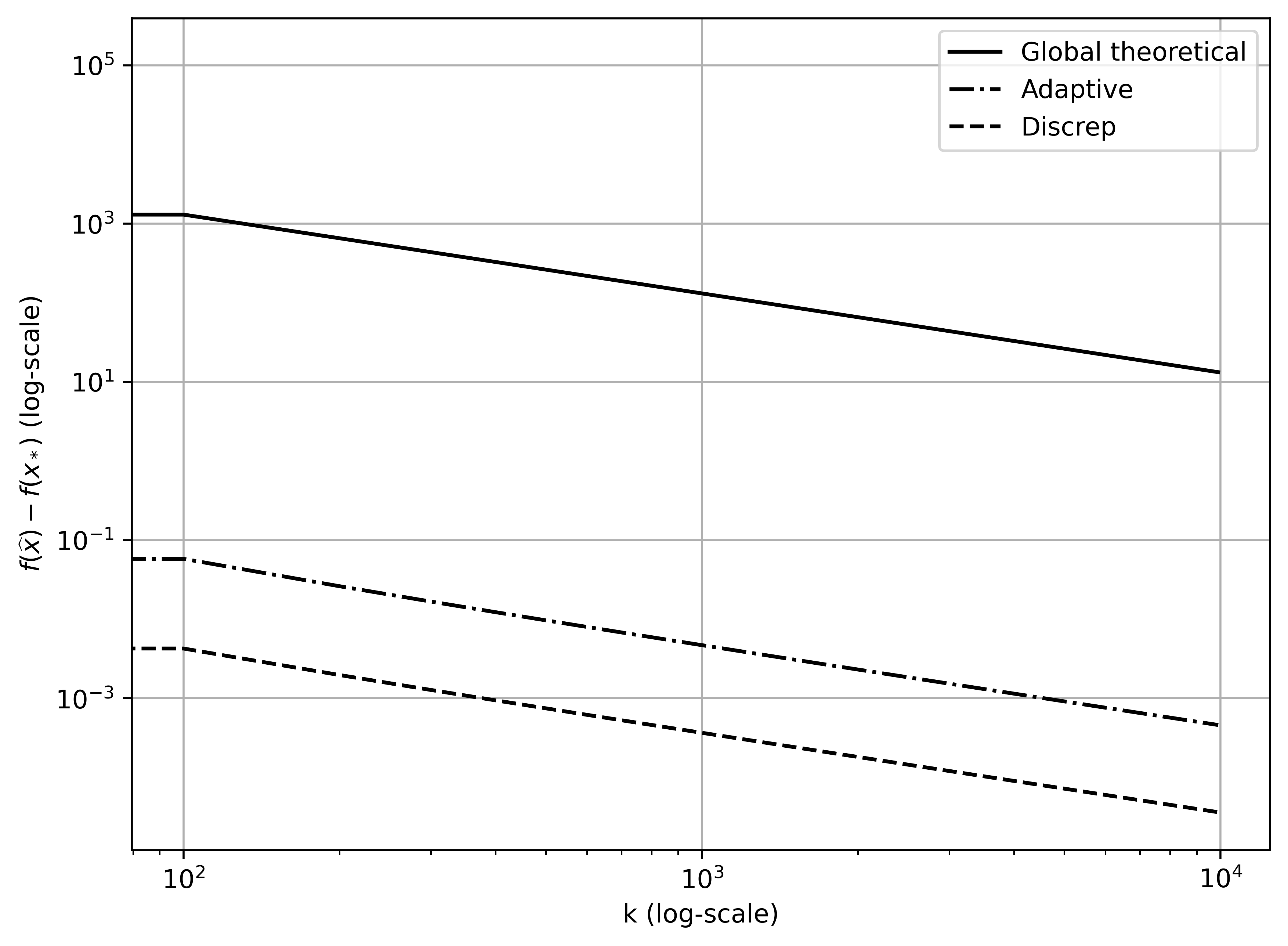}
	\endminipage\hfill
	\caption{ Результаты решения задачи минимизации (\ref{sphere_cover_strongly}), где  $n= 1\,000, r = 0.7525$.}
	\label{res_strong_convex}
\end{figure}

Тем не менее, сравнение с известным точным решением $x_*$, а также график динамики значения целевой функции показывает, что за малое число шагов (значительно меньшее, чем для метода \eqref{orig}) реализация метода \eqref{1} приводит к неплохому качеству приближённого решения. При этом, однако, для метода \eqref{1} после достижения такого уровня дальнейшее повышение качества выходной точки в отличие от метода \eqref{orig} уже не наблюдается.

\section{Другие вычислительные эксперименты для рассматриваемых методов с шагом типа Б.Т. Поляка}\label{experiments_sec}
Для демонстации эффективности предлагаемого алгоритма \eqref{1} с шагом \eqref{adaptive_step} (адаптивный вариант) и шагом \eqref{nonadaptive_step} (неадаптивный вариант) были также проведены некоторые численные эксперименты для сильно выпуклых задач 
и сравнены результаты их работы с предложенным в \cite{Bach_2012} субградиентным методом с шагом вида
\begin{gather}\label{step_size_bach}
    h_k = \frac{2}{\mu (k+1)} \quad \forall k \geqslant 0,
\end{gather}
где $\mu$ --- параметр сильной выпуклости целевой функции. Для алгоритма \eqref{step_size_bach} из \cite{Bach_2012} известна следующая оценка качества выходной точки после $N$ итераций:
\begin{gather}\label{estimate_bach}
 \|\widehat{x}-x_*\|_2^2 \leq \frac{4M^2}{\mu^2 (N+1)} \quad \forall k \geq 0.
\end{gather}
Отметим, что вместо $\widehat{x}$ можно выбрать $x_{min}$: $f(x_{min}) = \min_{k = \overline{1, N}} f(x_k)$. Если значения $f(x_k)$ убывают с ростом $k$, то в неравенстве \eqref{estimate_bach} можно $\widehat{x}$ заменить на $x_{N}$.

Эксперименты проводились для множества Q, равного шару с центром в $0 \in \mathbb{R}^n$ и радиусом $R$, т.е. $Q = B_R (0)=\{x \in \mathbb{R}^n : \|x\|_2 \leq R\}$. При этом рассмотрены 3 примера.

\begin{example}\label{ex_exact}
Пусть целевая функция имеет вид
\begin{gather}\label{objective_ex_exact}
	f(x) = \|x\|_2 + 2\gamma \|x\|_2^2, \quad \gamma >0.
\end{gather}
Такая функция $M$-липшицева с константой $M = 1 + 2 \gamma R$ и $2 \gamma$-сильно выпукла ($\beta = 1$ в \eqref{adaptive_step} и \eqref{nonadaptive_step}). Также отметим, что $f$ из \ \eqref{objective_ex_exact}  удовлетворяет обобщенному условию острого минимума \eqref{eq_gen_sharp} при $\alpha = 1, \Delta = 0, \overline{f} = f^* = 0 \in \mathbb{R}$ и $x_* = 0 \in Q$.
\end{example}

\begin{example}\label{ex_nonexact_exper}
В данном примере (см. пример \ref{ex_nonexact}) выберем негладкую функцию вида
\begin{gather}\label{obj_ex_nonexact}
f(x) = \|x\|_2 + 2\gamma \|x - c\|_2^2,
\end{gather}
которая является $M$-липшицевой с константой $M = 1 + 2 \gamma (R+ \|c\|_2)$  и $2 \gamma$-сильно выпуклой. При этом функция \eqref{obj_ex_nonexact} удовлетворяет обобщенному условию острого минимума \eqref{eq_gen_sharp} при $\alpha = 1, \Delta = 2 \gamma \|c\|_2^2$ и $\overline{f} = \gamma \|c\|_2^2$.
\end{example}

Cравниваемые алгоритмы были запущены для $n=1000$ и $10000$ со стартовой точкой $x_0 = \left ( \frac{R}{\sqrt{n}}, \ldots,  \frac{R}{\sqrt{n}} \right)$ и с разными значениями $R$ и $\gamma$. Результаты сравнения представлены на рисунках \ref{res_ex_exact_r10} и \ref{res_ex_exact_r1000} для примера \ref{ex_exact} и на рис. \ref{res_ex_nonexact} для примера \ref{ex_nonexact_exper}. \textit{ На этих рисунках жирная, пунктирная и штриховая кривые указывают на поведение адаптивной и неадаптивной версий алгоритма \eqref{1}, а также субградиентного метода \cite{Bach_2012} соответственно. Для примера \ref{ex_exact} мы отображаем динамику значений  $f(x_k) - f_* = f(x_k)$ и $\|x_k - x_*\|_2^2 = \|x_k\|_2^2$ в зависимости от номера итерации $k$. Для примера \ref{ex_nonexact_exper} ввиду того, что известно лишь $\overline{f}$, мы отображаем $f(x_k) - \overline{f}$, а также теоретические оценки \eqref{adaptive_estimate}, \eqref{nonadaptive_estimate}, \eqref{estimate_bach}, а также левую часть \eqref{arg_est} для иллюстрации качества решения, достигнутого сравниваемыми алгоритмами.}

\begin{figure}[htp]
	\minipage{0.49\textwidth}
	\includegraphics[width=\linewidth]{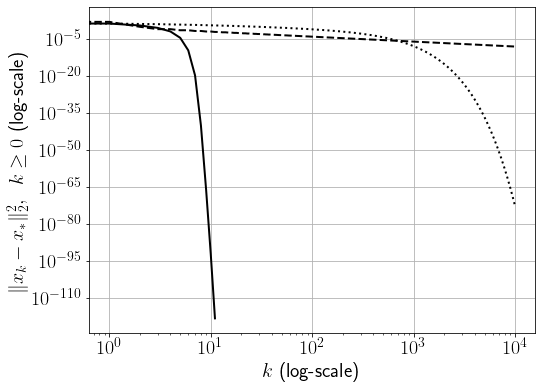}
	\endminipage\hfill
	\minipage{0.49\textwidth}
	\includegraphics[width=\linewidth]{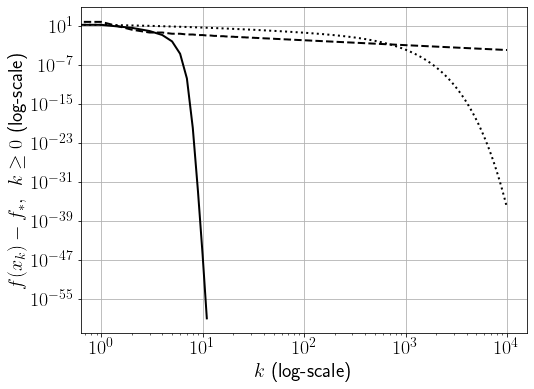}
	\endminipage\hfill
	\caption{ Результаты для примера \ref{ex_exact}, где  $n= 10\,000, R = 10$ и  $\gamma = 0.5$.}
	\label{res_ex_exact_r10}
\end{figure}

\begin{figure}[htp]
	\minipage{0.49\textwidth}
	\includegraphics[width=\linewidth]{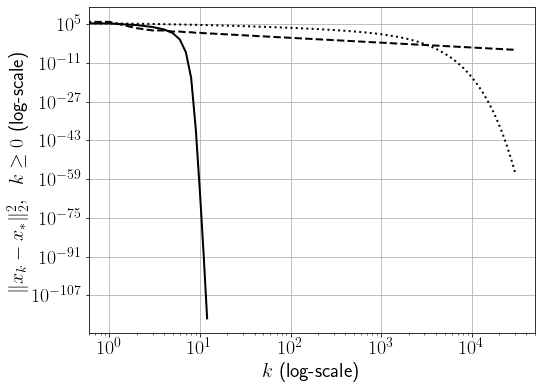}
	\endminipage\hfill
	\minipage{0.49\textwidth}
	\includegraphics[width=\linewidth]{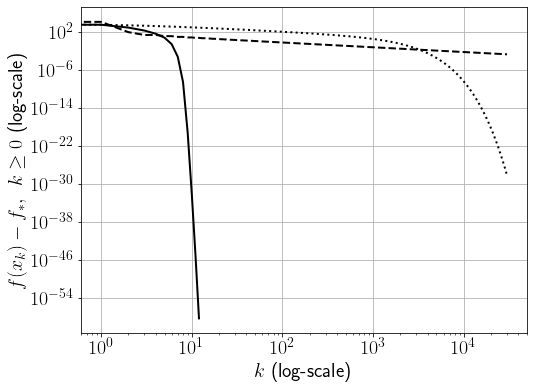}
	\endminipage\hfill
	\caption{Результаты для примера \ref{ex_exact}, где $n = 10\,000, R = 1000$ и  $\gamma = 0.01$.}
	\label{res_ex_exact_r1000}
\end{figure}

\begin{figure}[htp]
	\minipage{0.49\textwidth}
	\includegraphics[width=\linewidth]{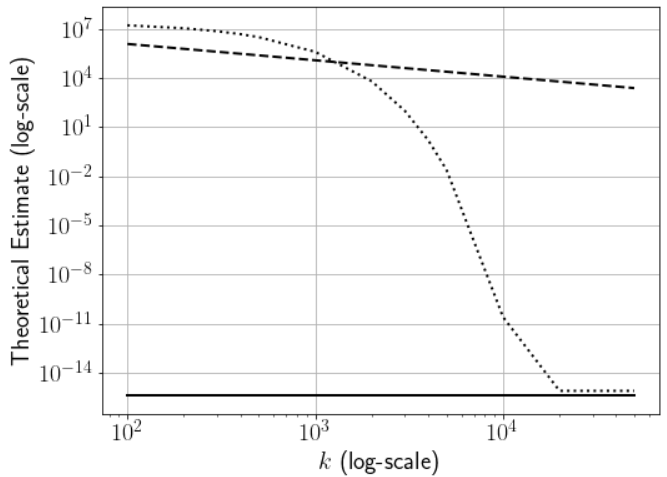}
	\endminipage\hfill
	\minipage{0.49\textwidth}
	\includegraphics[width=\linewidth]{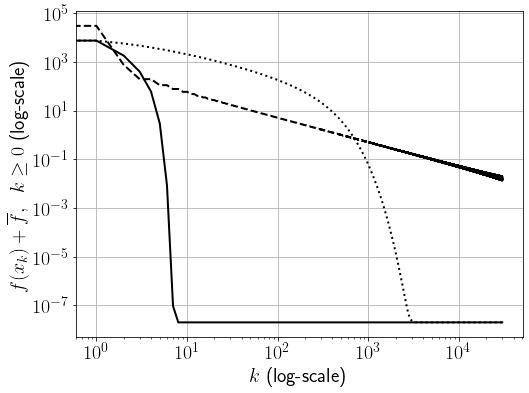}
	\endminipage\hfill
	\caption{Результаты для примера \ref{ex_nonexact_exper}, где $n=1000, R = 5000, \gamma = 0.001$ и $\|c\|_2 = 0.01$.}
	\label{res_ex_nonexact}
\end{figure}

Полученные результаты для примеров  \ref{ex_exact} и  \ref{ex_nonexact_exper} показывают хорошую эффективность предложенных в настоящей статье вариаций субградиентного метода с шагом Б.Т. Поляка. Как адаптивная, так и неадаптивная версии алгоритма \eqref{1} работают лучше по сравнению с предложенным в \cite{Bach_2012} субградиентным методом. Кроме того, мы видим, что адаптивная версия является лучшей с точки зрения скорости сходимости как по функции, так и по аргументу. Кроме того, предлагаемые в работе вариации алгоритма \eqref{1} с шагом Б.Т. Поляка (как в адаптивном, так и в неадаптивном случае) приводят к лучшему качеству точки выхода как с точки зрения теоретических оценок \eqref{adaptive_estimate} и \eqref{nonadaptive_estimate}, так и с точки зрения практической работы методов (см. рис. \ref{res_ex_nonexact}), если нет высоких требований к точности. При этом для вариации алгоритма \eqref{1} необходима информация о приближённом значении целевой функции $\overline{f}$ (работа методов может позволить получить соответсвующее качество точке выхода и <<по аргументу>>).

Также мы можем убедиться в эффективности предложенных в разделе 3 настоящей работе подходов в случае $\Delta$-острого минимума на следующем примере.
\begin{example}\label{ex_convex_hull}
Пусть целевая функция имеет следующий вид
\begin{gather}\label{obj_ex_convex_hull}
    F(x) = \min_{A\in S}\{\|x- A\|_2\} + \gamma \|x\|_2^2,
\end{gather}
где $S$ --- выпуклая оболочка двух шаров одного радиуса $r$, так что $S \subset Q$, функция $F$ является $M$-липшицевой с константой $M = 1 + 2 \gamma R$ и $2 \gamma$-сильно выпуклой (считаем, что $Q$ --- евклидов шар радиуса $R>0$). При этом функция \eqref{obj_ex_convex_hull} удовлетворяет обобщенному условию острого минимума \eqref{eq_gen_sharp} при $\alpha = 1, \Delta =  \gamma R^2$ и $\overline{F} = \gamma R^2$.
\end{example}

Результаты для примера \ref{ex_convex_hull} представлены на рисунках \ref{res1_ex_convex_hull} и \ref{res2_ex_convex_hull} для различных значений $\gamma, R$ и $r$, где штриховая черная линия указывает на оценку \eqref{estimate_bach}, а синяя --- на левую часть \eqref{arg_est}. Также показан характер поведения значений функции при работе сравниваемых методов. Подобное поведение фнкции при использовании \cite{Bach_2012} вызвано негладкостью рассматриваемой задачи, в оценках \eqref{adaptive_estimation_f} использование усреднения позволяет убрать такой эффект. На графиках справа из рисунках \ref{res1_ex_convex_hull} и \ref{res2_ex_convex_hull} линии, которые расположены внизу правых графиков соответствуют теоретическим оценкам для методов с шагом типа Б.Т. Поляка в разделе 3. Оценке \eqref{adaptive_estimate} соответсвует самая <<нижняя>> горизонтальная линия, а неадаптивной оценке \eqref{nonadaptive_estimate} --- параллельная ей пунктирная линия.

\begin{figure}[htp]
	\minipage{0.49\textwidth}
	\includegraphics[width=\linewidth]{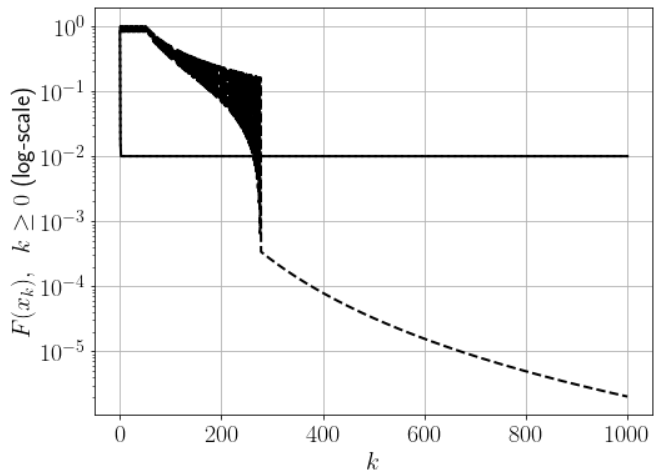}
	\endminipage\hfill
	\minipage{0.49\textwidth}
	\includegraphics[width=\linewidth]{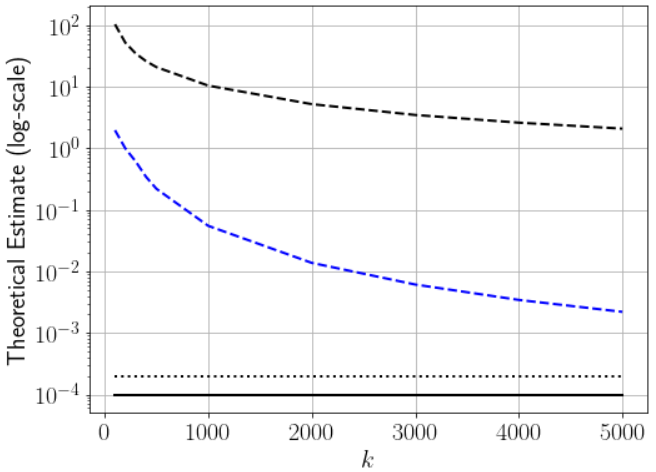}
	\endminipage\hfill
	\caption{Результаты для примера \ref{ex_convex_hull}, где $n=1000,  \gamma = 0.01, R = 1$ и $r = 0.1$.}
	\label{res1_ex_convex_hull}
\end{figure}
\begin{figure}[htp]
	\minipage{0.49\textwidth}
	\includegraphics[width=\linewidth]{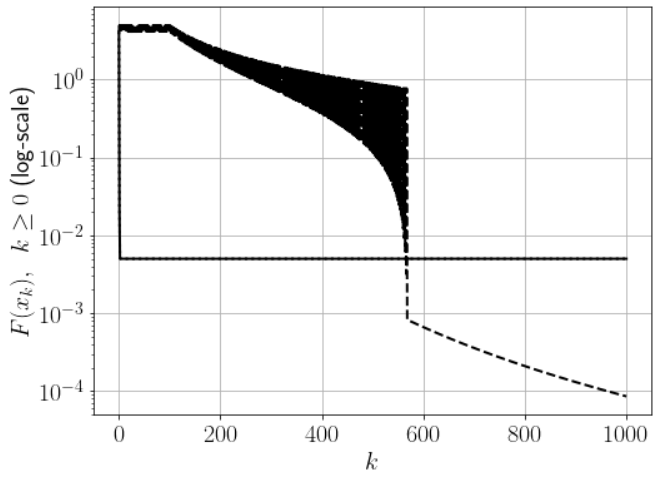}
	\endminipage\hfill
	\minipage{0.49\textwidth}
	\includegraphics[width=\linewidth]{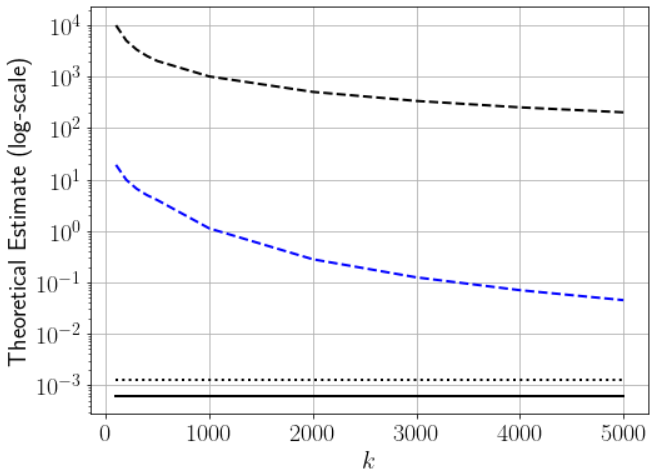}
	\endminipage\hfill
	\caption{Результаты для примера \ref{ex_convex_hull}, где $n=1000,  \gamma = 0.001, R = 5$ и $r = 0.5$.}
	\label{res2_ex_convex_hull}
\end{figure}

Из рисунков \ref{res1_ex_convex_hull} и \ref{res2_ex_convex_hull} можно сделать такой вывод: что алгоритм \eqref{step_size_bach} из \cite{Bach_2012} приводит к лучшим значения минимизируемой целевой функции $F$, но мы можем наблюдать эффективность предлагаемой в настоящей статье методики (раздел 3) в плане теоретической оценки с использованием $\Delta$-острого минимума.

\begin{remark}\label{remark2}
Экспериментально проверялась также работа предлагаемых алгоритмов для целевой функции вида
\begin{gather}\label{obj_more_balls}
    f(x) = \min_{A \in S} \{\|x - A\|_2\} \quad \forall x \in Q = B_R(0),
\end{gather}
где $S$ --- выпуклая оболочка $m$ шаров одного радиуса $r$, с центрами $O_i, i=1, \ldots, m$ не далеко друг от друга, причём $S \subset Q$. При этом функция  \eqref{obj_more_balls} $1$-липшицева и удовлетворяет условию острого минимума \eqref{sm} при $\alpha = 1$ и $f^* = 0$. При этом существенно, что направление спуска вычислялось неточно путём случайного выбора одного из центров $O_i$, то есть на итерациях использовались неточные $\delta$-субградиенты. Адаптивность метода тут не существенна, поскольку норма $\delta$-субградиента целевой функции равна 1 в точке, отличной от точного решения. Оказалось, что после нескольких итераций (в среднем 5 итераций) методы останавливаются и выдают точное решение \eqref{obj_more_balls}.
\end{remark}

\begin{remark}\label{remark3}
Экспериментально проверялась также работа предложенной методики для примера \ref{example1}. Поскольку в данном случае известно, что $M = 1$, то нет смысла отдельно выделять адаптивный шаг. Расчёты выполнены при $\Delta = 0.5, n=1000$ для $m = 20$ шаров $\Omega_i (i=1,\ldots, m)$ радиуса $r = 10$. Центры шаров $O_i$ различны и подобраны так, что расстояние от всех $O_i$ до нуля $0 \in \mathbb{R}^n$ равно $\Delta + r$. По итогам проведённых расчётов за малое количество итераций (около 20) удалось найти лучшую по сравнению со стартовой точку, удалённую от всех шаров не более, чем $0.407$, что несколько лучше $\Delta$. Стоит отметить, что для данного примеры были использована вариация шага Б.Т. Поляка \eqref{nonadaptive_step} с выбором параметров для подзадач $\overline{f}_i \leq \Delta$ без требования $\overline{f}_i \geq f^*_i$, хотя теоретические результаты настоящей статьи получены при допущении $\overline{f} \geq f^*$.
\end{remark}

\section{Заключение}

В статье рассмотрены некоторые подходы к методам для задач негладкой минимизации с оценками качества выдаваемого решения, инвариантными по размерности пространства, что потенциально интересно для задач в пространствах больших размерностей. Большая часть работы посвящена новым результатам по методам с <<неточными>> аналогами шага Б.Т. Поляка \eqref{adaptive_step}, \eqref{nonadaptive_step}, \eqref{step_teor4}, а также <<неточным>> аналогом условия острого минимума \eqref{eq_gen_sharp}. Такой подход даёт хорошие оценки скорости сходимости при достаточно малом $\Delta \geq 0$ как для выпуклых, так и для некоторых типов задач с релаксациями выпуклости. Приведено несколько примеров постановок задач не обладающих обычным острым минимумом, но для которых может быть выполнено его обобщение \eqref{eq_gen_sharp}. Тем самым, предлагаемый подход несколько расширяет класс применимости методов с шагом Б.Т. Поляка и обычным острым минимумом \cite{6}. Более того, полученные в настоящей работе адаптивные и частично адаптивные оценки качества приближённого решения, выдаваемого предлагаемыми вариантами субградиентного метода, потенциально приводят к возможности использовать такие методы и теоретические оценки для нелипшицевых задач (например, для минимизации локально липшицевых функций, удовлетворяющих условию Гёльдера) или для задач с неизвестной константой Липшица $M$. Однако c неизвестной константой Липшица лишь о применимости оценки \eqref{adaptive_estimate}. Тем не менее, мы приводим пример вычислительных экспериментов\footnote{Все эксперименты были реализованы на Python 3.4 и пакета numpy, при помощи стандартного представления чисел в формате np.float64.} для сильно выпуклой задачи с острым минимумом, где предложенная в разделе 3 данная методика себя оправдывает по сравнению с известным методом \cite{Bach_2012}. Попутно также выведена оценка качества приближённого решения, выдаваемого субградиентным методом \cite{Bach_2012} с адаптивно подбираемыми параметрами. Выполнено экспериментальное сравнение указанных двух подходов для задачи о наименьшем покрывающем шаре. Оно показало, что метод с шагом Б.Т. Поляка может за малое число шагов приводить к относительно неплохому уровню точности для некоторых вариантов датасета (набора точек) в силу того, что учитывает геометрию задачи. Но при этом даже при специальных допущениях удаётся показать лишь условие $\Delta$-острого минимума при некотором $\Delta > 0$, что не позволяет гарантировать глобальной сходимости метода со скоростью геометрической прогрессии. В плане возможных задач на будуще можно было бы отметить актуальность  разработки рандомизированных вариантов исследованных в статье методов, а также более детальное исследование возможности замены предположения $\overline{f} \geq f^*$ менее жёсткими требованиями для теоретических оценок и практической реализации методов.


\newpage


\end{document}

%% file: header.tex
\usepackage[utf8]{inputenc} 
\usepackage[T1]{fontenc}    
\usepackage{hyperref}       
\usepackage{url}            
\usepackage{booktabs}       
\usepackage{amsfonts}       
\usepackage{nicefrac}       
\usepackage{lipsum}
\usepackage{amsmath,amssymb,amsthm}
\usepackage{mathtools}
\usepackage{algorithm, algorithmic}
\usepackage{pifont}
\usepackage{cases}
\usepackage{graphicx}
\usepackage{stackengine}    
\usepackage{titlesec}
\usepackage[none]{hyphenat}
\usepackage{subfigure}
\usepackage[labelsep=period]{caption}
\usepackage{hyperref}
\hypersetup{
  colorlinks   = true,    
  urlcolor     = green,   
  linkcolor    = blue,    
  citecolor    = red      
}

\theoremstyle{plain}
\newtheorem{theorem}{Теорема}
\newtheorem{lemma}{Лемма}

\newtheorem{corollary}{Следствие}

\theoremstyle{definition}
\newtheorem{remark}{Замечание}
\newtheorem{example}{Пример}

\def\keywordname{{\bfseries Ключевые слова:}}%
\def\keywords#1{\par\addvspace\medskipamount{\rightskip=0pt plus1cm
\def\and{\ifhmode\unskip\nobreak\fi\ $\cdot$ }\noindent\keywordname\enspace\ignorespaces#1\par}}

\def\keywordnameeng{{\bfseries Keywords:}}%
\def\keywordseng#1{\par\addvspace\medskipamount{\rightskip=0pt plus1cm
\def\and{\ifhmode\unskip\nobreak\fi\ $\cdot$ }\noindent\keywordnameeng\enspace\ignorespaces#1\par}}

\def\subjclassname{{\bfseries Mathematics Subject Classification (2000):}}%
\def\subjclass#1{\par\addvspace\medskipamount{\rightskip=0pt plus1cm
\def\and{\ifhmode\unskip\nobreak\fi\ $\cdot$ }\noindent\subjclassname\enspace\ignorespaces#1\par}}

\newcommand*{\No}{\textnumero}

\titlelabel{\thetitle. }

%% file: main.bbl
\begin{thebibliography}{99}

\bibitem[Нестеров, 1989]{3}
{\it Нестеров~Ю.~Е.}
Эффективные методы нелинейного программирования.~--- М.: Радио и связь, 1989.~--- 301~с.\\
{\footnotesize{\it Nesterov~Yu.} Effektivnye metody nelineinogo programmirovaniya [Efficient Nonlinear Programming Methods]~// M.: Radio i Svyaz’ Publ., 1989.~--- 301~p. ISBN: 5-256-00524-3. (in Russian). \par}

\bibitem[Нестеров, 2010]{9}
{\it Нестеров~Ю.~Е.}
Методы выпуклой оптимизации.~--- М.: МЦНМО. 2010.~--- 281~с.\\
{\footnotesize{\it Nesterov~Yu.} Metody vypukloi optimizatsii [Convex optimization methods]~// Moscow: Publishing House MCNMO, 2010.~--- 281~p. (in Russian). ISBN: 978-5-94057-623-5. \par}

\bibitem[Поляк, 1969]{6}
{\it Поляк~Б.~Т.}
Минимизация негладких функционалов //  Журн. вычисл. математики и мат. физики.~---1969.--- Т.~9, \No~3.,~---C.~509--521. DOI: 10.1016/0041-5553(69)90061-5.\\
{\footnotesize{\it Polyak~B.~T.} Minimizatsiya negladkikh funktsionalov [Minimization of Nonsmooth Functionals]~// Zhurn. vychisl. matematiki i mat. fiziki.~--- 1969.~--- Vol.~9, No.~3.~--- S.~509--521 (in Russian).\par}

\bibitem[Polyak, 1987]{1}
{\it Поляк~Б.~Т.}
Введение в оптимизацию. —-- М: Наука, 1983. --- 384 c.
{\footnotesize{\it Polyak~B.~T.}
Introduction to optimization.// Optimization software. Inc., Publications Division, New York 1, 1987. --- xxvii+438 pp. \par}


\bibitem[Стонякин, 2020]{10}
{\it Стонякин~Ф.~С.}
Адаптивные зеркальные спуски для задач выпуклого программирования с использованием $\delta$-субградиентов// arXiv preprint ~--- 2020.
\url{https://arxiv.org/pdf/2012.12856.pdf}
\\
{\footnotesize{\it Stonyakin~F.~S.} Adaptivnye zerkal'nye spuski dlya zadach vypuklogo programmirovaniya s ispol'zovaniem $\delta$-subgradientov [Adaptive Mirror Descent Methods for Convex Programming Problems with $delta$-subgradients]~// arXiv preprint ~--- 2020. Available at:
\url{https://arxiv.org/pdf/2012.12856.pdf} (in Russian).\par}

\bibitem[Стонякин et al., 2021]{5}
{\it Стонякин~Ф.~С., Баран~И.~В., Аблаев~С.~С.}
Адаптивные методы градиентного тима для задач минимизации с относительной точностью и острым минимумом // Труды ИММ УрО РАН~---2021.--- Т.~27, \No~4,~---С.~175--188. DOI: 10.21538/0134-4889-2021-27-4-175-188. \\
{\footnotesize{\it Stonyakin~F.~S., Baran~I.~V., Ablaev~S.~S.} Adaptivnye metody gradientnogo tima dlya zadach minimizatsii s otnositel'noi tochnost'yu i ostrym minimumom [Adaptive gradient-type methods for optimization problems with relative error and sharp minimum]~// Trudy IMM UrO RAN~--- 2021.~--- Vol.~27, No.~4.~--- S.~175--188 (in Russian).  DOI: 10.21538/0134-4889-2021-27-4-175-188. Available at: \url{https://arxiv.org/pdf/2103.17159.pdf}.\par}

\bibitem[Devolder et al., 2014]{DevGleenerNesterov}
{\it Devolder~O., Glineur~F., Nesterov~Yu.}
First-order methods of smooth convex optimization with inexact oracle // Math. Programming.~---2014.--- Vol.~146.~--- No.~1.~---P.~37--75.
DOI: 10.1007/s10107-013-0677-5.

\bibitem[Hardt et al., 2018]{7}
{\it Hardt~M., Ma~T., Recht~B.}
Gradient descent learns linear dynamical systems // J. Mach. Learn. Res.,~---2018.--- Vol.~19, No.~29, 44~p.

\bibitem[Hinder et al., 2020]{8}
{\it Hinder~O., Sidford~A., Sohoni~N.~S.}
Near-optimal methods for minimizing star-convex functions and beyond //  Proceedings of Machine Learning Research~---2020.--- Vol.~125,~---P.~1894--1938.

\bibitem[Julien et al., 2012]{Bach_2012}
{\it Julien~S.~L.,  Schmidt~M.,  Bach~F.}
A simpler approach to obtaining an $O(1/t)$ convergence rate for the projected stochastic subgradient method // arXiv preprint~--- 2012. Available at: \url{https://arxiv.org/pdf/1212.2002.pdf}.


\bibitem[Stonyakin et al., 2020]{11}
{\it Stonyakin F.S., Stepanov A.N., Gasnikov A.V., Titov A.A.}
Mirror descent for constrained optimization problems with large subgradient values of functional constraints // Computer Research and Modeling~--- 2020.--- Vol.~12, No.~2.~--- P.~301--317. DOI: 10.20537/2076-7633-2020-12-2-301-317. Available at: \url{https://arxiv.org/pdf/1908.00218.pdf}.

\bibitem[Stonyakin et al., 2021b]{Stonyakin_2021}
{\it Stonyakin~F.~S., Titov~A.~A., Makarenko~D.~V.,  Alkousa~M.~S. }
Some Methods for Relatively Strongly Monotone Variational Inequalities // arXiv preprint~--- 2021. Available at: \url{https://arxiv.org/pdf/2109.03314.pdf}.

\end{thebibliography}
